\documentclass{article}
\usepackage[utf8]{inputenc}

\usepackage{amsmath,amsthm,verbatim,amssymb,amsfonts,amscd, graphicx, tikz-cd, pinlabel, thm-restate, xcolor}
\usepackage{graphics}

\theoremstyle{plain}
\newtheorem{theorem}{Theorem}

\newtheorem{lemma}{Lemma}

\newtheorem{proposition}{Proposition}

\theoremstyle{definition}
\newtheorem{definition}{Definition}
\newtheorem*{theorem*}{Theorem}
\newtheorem*{corollary*}{Corollary}
\newtheorem*{proposition*}{Proposition}

\title{Uniqueness of 4-manifolds described as sequences of 3-d handlebodies}
\author{Gabriel Islambouli}
\date{}

\begin{document}

\maketitle

\section{Introduction}

Given a description of a manifold, two natural questions arise: existence and uniqueness. For example, closed, orientable 3-manifolds can be described as surgeries on a link in $S^3$. The existence problem for this description was addressed by Lickorish and Wallace \cite{WL} \cite{AW} who showed that any closed, orientable 3-manifold is given by surgery on a link in $S^3$. In \cite{Kir78},  Kirby proved a uniqueness theorem for these descriptions, showing that any two link-surgeries yielding the same 3-manifold are related by handle slides and blowups. Similarly, Moise \cite{EM} showed that every closed, orientable 3-manifold can be described as a Heegaard splitting, and Reidemeister and Singer \cite{JS} \cite{KR} showed that any two Heegaard splittings giving the same 3-manifold are related by a stabilization operation.

Moving up a dimension, there has been a variety of new descriptions of smooth, orientable, closed 4-manifolds. In \cite{Gay19}, Gay showed that any such 4-manifold can be described as a loop of Morse functions on a surface, and used this fact to give a novel proof of that the smooth, oriented, 4-dimensional cobordism group is isomorphic to $\mathbb{Z}$. In \cite{IK}, Klug and the author showed that these 4-manifolds can be represented as a loop in the pants complex, and this was used to provide a combinatorial calculation of the aforementioned cobordism group as well as to give invariants of loops in the pants complex. Kirby and Thompson \cite{KT} have also showed that any such 4-manifold can be described as a loop in the cut complex, and they used this to define an invariant of 4-manifolds which detects the 4-sphere among homotopy spheres. Finally, in \cite{IN20} Naylor and the author introduced multisections of 4-manifolds and used them to describe 4-manifold operations such as cork twisting and log transforms diagrammatically on a surface.

As these descriptions all stem from trisections \cite{GK}, there are inherent similarities between them. Most obviously, all of the information in these decompositions can be given on a surface. Nevertheless it was not clear how to pass between these descriptions. The first half of this paper is focused on constructing explicit correspondences between certain quotients of the aforementioned sets. These quotients still retain all of the information needed to construct a unique smooth 4-manifold, together with a multisection structure. 

The second half of the paper is dedicated to showing how any two multisections yielding diffeomorphic smooth 4-manifolds are related. To do this, we diagrammatically realize a move introduced in \cite{IN20} called a UPW move, as it is a combination of previously defined unsink, push, and wrinkle moves. The unsink portion of this move produces a Lefschetz singularity and the total move modifies a multisection diagram in the neighbourhood of this Lefschetz singularity as in Figure \ref{fig:wrinklingAnnulus}. Passing through the correspondences constructed in the first section, we are able to realize a UPW move as a modification of generic loops of Morse functions, a modification of loops in the pants complex, and as a modification of loops in the cut complex. This leads to the main theorem of the paper, stated in different notation in Theorem \ref{thm:secondBijections}.

\begin{theorem*}
There are bijections between the following sets:
\begin{itemize}

\item Smooth, orientable, closed 4-manifolds modulo diffeomorphism.
\item Loops in the cut complex of $\Sigma_g$ modulo type 0 moves, commuting type-1 moves, UPW moves, and an overall diffeomorphism of $\Sigma_g$.
\item Loops in the pants complex of $\Sigma_g$ modulo A-moves, commuting S-moves, UPW moves, and an overall diffeomorphism of $\Sigma_g$.
\item Generic loops of Morse functions $f_t: \Sigma_g \to \mathbb{R}$ modulo generic paths with one critical value crossing in a torus, commuting crossings in tori, UPW moves, and an overall diffeomorphism of $\Sigma_g$.
\item Multisections of 4-manifolds modulo UPW moves, expansions, contractions, and multisection diffeomorphism.

\end{itemize}
\end{theorem*}

To prove this theorem, we use the fact shown in \cite{IN20} that repeated application of this move reduces a multisection into a trisection. Once the multisection has been reduced to a trisection, we show how this move encompasses moves sufficient to relate any two trisections of the same 4-manifold. In particular the following proposition is proven diagrammatically in Figures \ref{fig:sildesToStabilization} and \ref{fig:remainingCurves} .

\begin{proposition*}
Trisection stabilization can be realized as a sequence of three UPW moves.
\end{proposition*}

Along the way to the main uniqueness theorem, we prove correspondences between the unstable objects involved in these descriptions of 4-manifolds. In particular, in Section \ref{sec:firstBijections} we prove the following correspondences by giving explicit maps between each of the sets.

\begin{theorem*}
Let $\Sigma_g$ be a closed, orientable surface of genus $g$. There are bijections between the following sets:

\begin{itemize}

\item Loops in the cut complex of $\Sigma_g$ modulo type-0 moves.
\item Loops in the pants complex of $\Sigma_g$ modulo A-moves.
\item Generic loops of Morse functions $f_t: \Sigma_g \to \mathbb{R}$ modulo generic paths which  have at most one critical value exchange in a torus.
\item Thin multisections with multisection surface $\Sigma_g$ modulo multisection diffeomorphism.

\end{itemize}

\end{theorem*}

In particular, the bijection between loops in the cut complex and generic loops of Morse functions gives the following corollary to Proposition \ref{prop:MorseToComplexInjection}, which may be of independent interest.

\begin{corollary*}
Let $f_0: \Sigma_g \to \mathbb{R}$ and $f_1: \Sigma_g \to \mathbb{R}$ be two Morse functions which have a collection of level sets $C_0$ and $C_1$, respectively, which form slide equivalent cut systems of $\Sigma_g$. Then $f_0$ and $f_1$ are connected by a generic path of Morse functions, $f_t$, such that $f_t$ contains no critical value crossings in a torus.
\end{corollary*}

\section{Preliminary definitions and constructions}
\subsection{Constructing 4-manifolds and multisections from sequences of handlebodies}
\label{sec:constructingMultisections}
\label{sec:4mfldFromHB}
Throughout this paper, we let $\Sigma_g$ be a fixed closed orientable surface of genus $g$. We denote by $S_{g,b}$ some surface homeomorphic to a surface of genus $g$ with $b$ boundary components. Many of the constructions in this paper take in an object associated to $\Sigma_g$ and extract a sequence of handlebodies bounded by $\Sigma_g$. Adjacent handlebodies in these sequences are closely related, in a way made precise by the following definition.

\begin{definition}
Handlebodies $H_1$ and $H_2$ with boundary $\Sigma_g$ are said to be \textbf{dual handlebodies} if the Heegaard splitting $H_1 \cup_{\Sigma_g} H_2$ is a Heegaard splitting of $\#^{k} S^1 \times S^2$, for some integer $k$. $H_1$ and $H_2$ are said to be \textbf{minimally dual} if they form the Heegaard splitting of $\#^{g-1} S^1 \times S^2$. A \textbf{cyclic sequence of dual handlebodies} is a cyclically ordered list $L = (H_1,...H_n)$ such that $H_i$ and $H_{i+1}$ are dual handlebodies for $i \in \mathbb{Z}$ (with indices taken mod $n$).
\end{definition}

By theorems of Waldhausen \cite{FW} and Haken \cite{WH}, cut systems for a pair of minimally dual handlebodies can be made to look exactly like those in Figure \ref{fig:DualHandlebodies}, after some handle slides and a homeomorphism of $\Sigma_g$. Following Gay and Kirby in \cite{GK}, a sequence of dual handlebodies gives rise to a unique smooth 4-manifold as described in the definition below, and illustrated in Figure \ref{fig:FillInTrisection}.

\begin{figure}
    \centering
    \includegraphics[scale=.2]{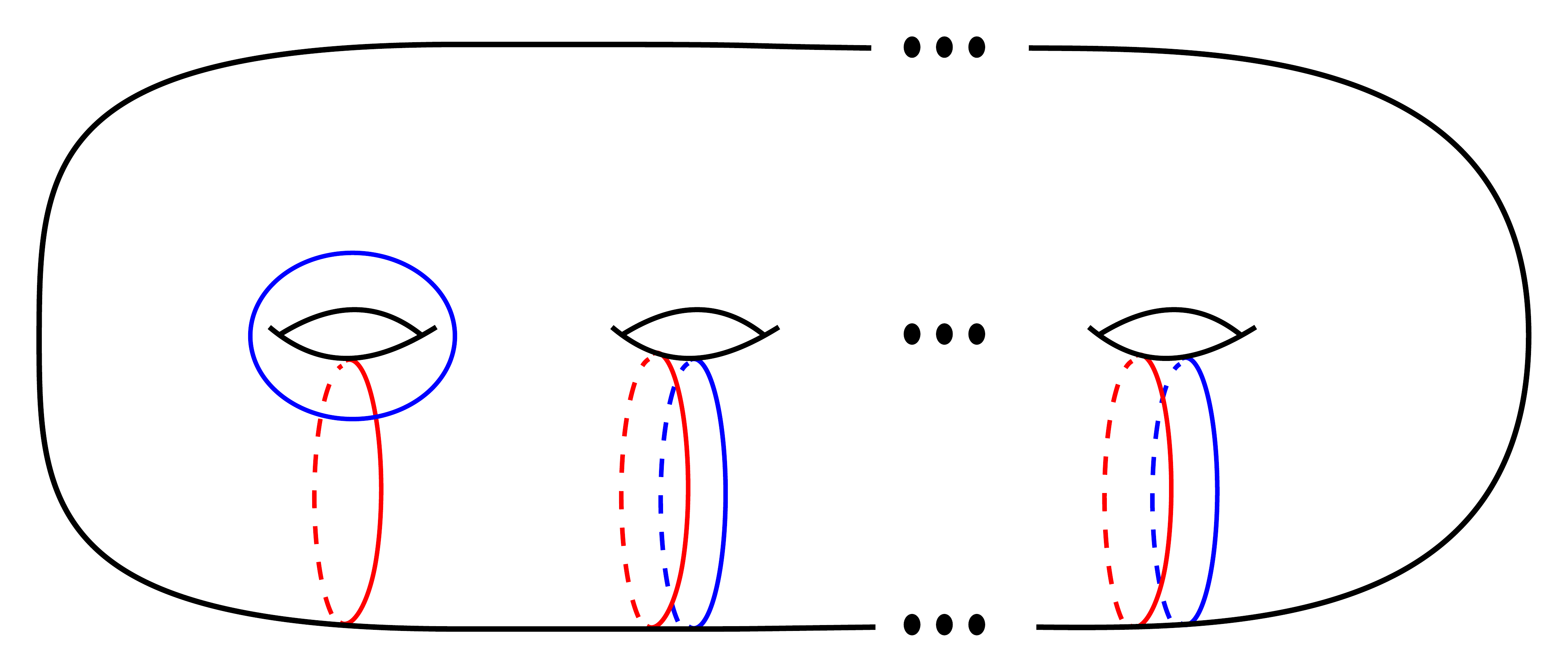}
    \caption{Cut systems for minimally dual handlebodies.}
    \label{fig:DualHandlebodies}
\end{figure}

\begin{definition}
\label{def:MfldFromHBs}
Let $L = (H_1, \cdots, H_n)$ be a cyclically ordered list of handlebodies with boundary $\Sigma$. Then the \textbf{smooth 4-manifold associated to L}, denoted $\mathbf{X^4(L)}$, is the 4-manifold obtained by the following process: 
\begin{itemize}
    \item Start with $\Sigma\times D^2$ and view the boundary $\Sigma \times S^1$ as $\Sigma \times [0,1]$ with $\Sigma \times \{0\}$ identified with $\Sigma \times \{1\}$
    
    \item Attach $H_i \times [i/n, \frac{i+1/2}{n}]$ to  $\Sigma\times D^2$ by thickening the map specifying the attachment of $H_i$ to $\Sigma$.
    
    \item Since $H_i$ and $H_{i+1}$ are dual, the resulting 4-manifold has $n$ boundary components, each diffeomorphic to $\#^{k_i}S^1\times S^2$. Cap each of these boundary components off with $\sharp^k S^1 \times B^3$  uniquely \cite{LP} to produce the smooth closed 4-manifold $X^4(L)$.
\end{itemize}
\end{definition}

\begin{figure}
    \centering
    \includegraphics[scale=.3]{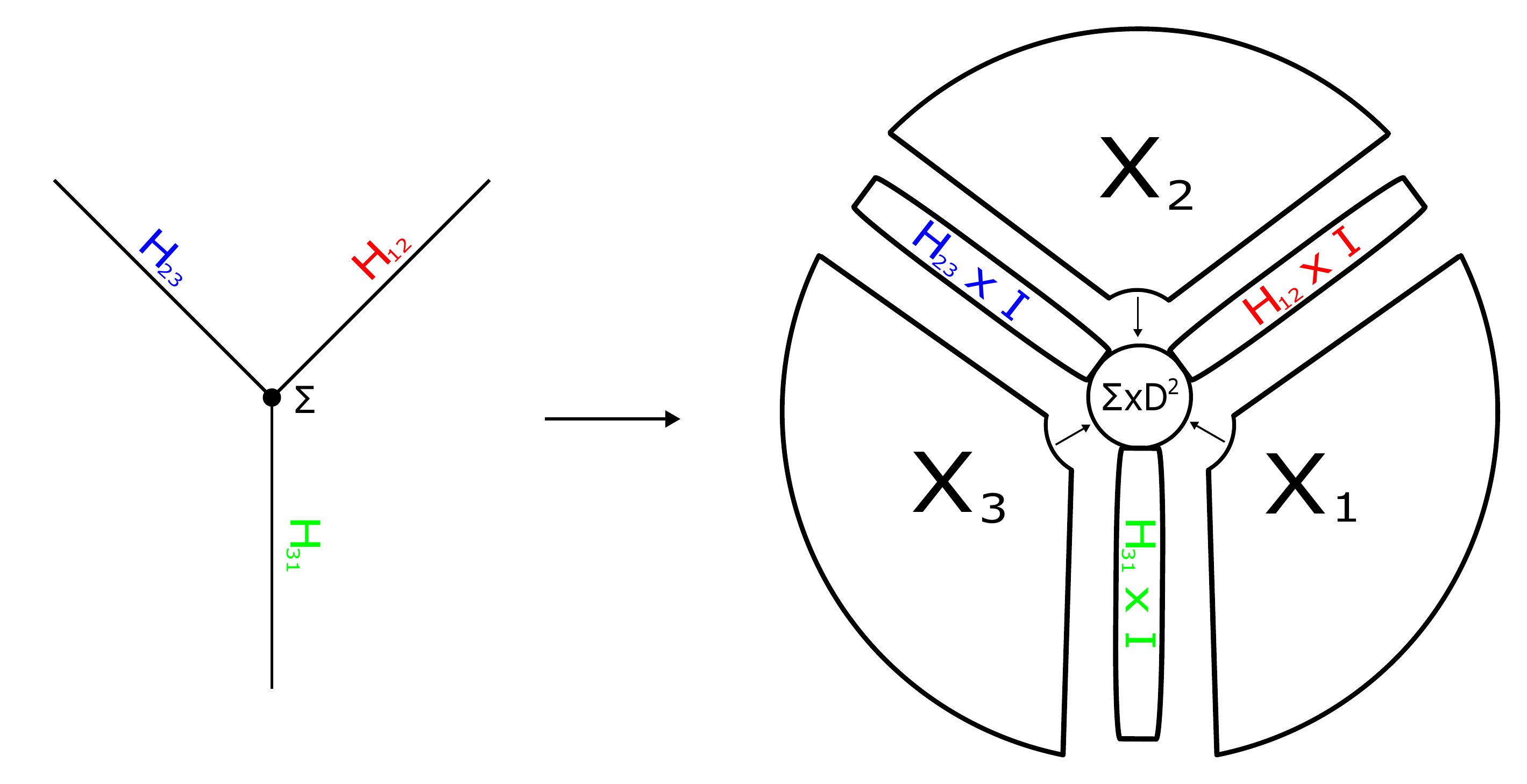}
    \caption{The 4-manifold associated to a sequence of dual handlebodies can be obtained by thickening the surface and handlebodies to 4-dimensions, then filling in the resulting boundary components uniquely with copies of $\sharp^k S^1 \times B^3$. This manifold inherits a multisection structure given by the $X_i$.}
    \label{fig:FillInTrisection}
\end{figure}

Given a loop of dual handlebodies $L$, we will see that the 4-manifold $X^4(L)$ comes naturally equipped with the structure of a multisection. These generalizations of trisections were defined and studied in \cite{IN20}, and we reproduce the definition here.

\begin{definition}
\label{def:multisection}
Let $X$ be a smooth, orientable, closed, connected 4-manifold. An $n$\emph{-section}, or \emph{multisection} of $X$ is a decomposition $X = X_1 \cup X_2 \cup \cdots \cup X_n$ such that:
\begin{enumerate}
    \item $X_i \cong \natural^{k_i} S^1 \times B^3$;
    \item $X_1 \cap X_2 \cap \dots \cap X_n = \Sigma_g$, a closed orientable surface of genus $g$ called the \textbf{multisection surface};
    \item $X_i\cap X_j=H_{i,j}$ is a 3-dimensional 1-handlebody if $|i-j|=1$, and $X_i \cap X_j = \Sigma_g$ if $|i-j|>1$;
    \item $\partial X_i \cong \#^{k_i} S^1 \times S^2$ has a Heegaard splitting given by $H_{(i-1),i} \cup_{\Sigma} H_{i,(i+1)}$.
\end{enumerate}

The $X_i$ are called \textbf{sectors} of the multisection. If each sector is diffeomorphic to $\natural^{g-1} S^1 \times B^3$, then the multisection is called \textbf{thin}.
\end{definition}

There are various equivalence relations one could consider on a multisection, but in this paper, we will primarily be using the following. 

\begin{definition}
Let $X$ and $Y$ be 4-manifolds and let $X= X_1 \cup ... \cup X_n$ and $Y = Y_1 \cup ... \cup Y_n$ be multisections. A \textbf{diffeomorphism of multisections} is a diffeomorphism $\phi: X \to Y$ such that $\phi (X_i) = Y_i$. We denote by $\mathcal{T}(\Sigma)_g$ the set of diffeomorphism classes of thin multisections whose multisection surface is $\Sigma_g$.
\end{definition}

Given a loop $L = (H_1, ... ,H_n)$ of dual handlebodies, the multisection structure on $X^4(L)$ is given by $X^4(L) = X_1 \cup X_2 \cup \cdots \cup X_n$ where $X_i$ is the union of $H_i$, $H_{i+1}$ and the copy of $\natural^{k_i} S^1 \times B^3$ filling them in. If $L$ is a sequence of minimally dual handlebodies, then the induced multisection is thin. We call this multisection structure $\textbf{T(L)}$, and we will show that it is well defined up to diffeomorphisms of multisection in Section \ref{sec:DHtoMS}.

We also note that the 4-manifold, together with the multisection structure, can be reconstructed from an ordered sequence of cut systems for the handlebodies $H_i$ drawn on $\Sigma_g$. This is called a \textbf{multisection diagram} and we refer the reader to the top and bottom of Figure \ref{fig:S2xS2QuadtoTri} for some basic examples and to $\cite{IN20}$ for more involved examples.

\subsection{Morse functions on a surface}

In this paper, we require that Morse functions be stable so that, in particular, each critical value has at most one critical point in its preimage. In this section, following \cite{Gay19}, we will show how to extract a sequence of dual handlebodies from a generic path of Morse functions on $\Sigma_g$. We begin by associating a handlebody to a Morse function.

\begin{definition}
Let $f: \Sigma \to \mathbb{R}$ be a Morse function. We define the \textbf{handlebody associated to $f$}, denoted $\mathbf{H(f)}$ to be the unique handlebody where for each regular value, $r$, every component of $f^{-1}(r)$ bounds a disk.
\end{definition}

In addition to the data of the handlebody $H(f)$, a Morse function also gives rise to a family of cut systems for the handlebody $H(f)$. In fact, a Morse function most naturally gives rise to a \textbf{pants decomposition}, which is a collection of $3g-3$ disjoint, non-isotopic, essential, simple closed curves.

\begin{definition}
\label{def:PfandCf}
Let $f: \Sigma \to \mathbb{R}$ be a Morse function. We define the \textbf{pants decomposition associated to $f$}, denoted $\mathbf{P(f)}$ to be the unique pants decomposition where every curve is isotopic to a component of $f^{-1}(r)$ for some regular value $r$. We define a \textbf{cut system associated to $f$}, denoted $\mathbf{C(f)}$ to be a cut system where every curve is isotopic to a component of $f^{-1}(r)$ for some regular value $r$.
\end{definition}

Any two Morse functions $f_0$ and $f_1$ on $\Sigma_g$ are connected by a 1-parameter family of functions, $f_t$ which is Morse except for at a finite number of  \textit{critical times} where the function is ``near Morse" \cite{JC}.  In these paths, at any critical time, $c$, $f_c$ has either two critical points mapping to the same critical value, which we call a \textbf{critical value exchange}, or has a single birth or death singularity. We call a path containing only such critical times a \textbf{generic path of Morse functions}.

Let $f_t$ be a generic path of Morse functions and let $c$ be a critical time. Since critical times are isolated, we can take $\epsilon$ be sufficiently small so that  there are no other critical times in $[c-\epsilon, c+\epsilon]$. We seek to understand how the handlebodies associated to the functions change as we pass a critical time. Note that since a handlebody is completely determined by a cut system, $H(f_t)$ is determined by $C(f_t)$. Since birth and death singularities only introduce or remove  null-homotopic level sets, if $c$ is a birth or death time, $H(f_{c- \epsilon}) = H(f_{c+ \epsilon}).$ 
 If $c$ is a critical value exchange, then, by Euler characteristic considerations, $f^{-1}([c-\epsilon, c+\epsilon])$ contains some number of annuli, which can be ignored, together with either two disjoint copies of $S_{0,3}$, a copy of $S_{0,4}$, or a copy of $S_{1,2}$.

In the case that the two critical points lie two disjoint copies of $S_{0,3}$, $P(f_{c-\epsilon}) = P(f_{c+\epsilon})$ so that the handlebodies do not change. Since we can select a cut system disjoint from a $S_{0,4}$ subsurface, if the saddles at $c$ occur in $S_{0,4}$ then again $H(f_{c- \epsilon}) = H(f_{c+ \epsilon}).$ On the other hand, if the saddles at $c$ occur in $S_{1,2}$, then up to homeomorphism, the level sets of $f_{c- \epsilon}$ and $f_{c+ \epsilon}$ are the curves at the left and the right of Figure \ref{fig:6ASPath}. Since in this case we can choose $C(f_{c- \epsilon})$ and $C(f_{c + \epsilon})$ to be identical, aside from two curves which intersect once, $H(f_{c- \epsilon})$ and $H(f_{c+ \epsilon})$ are minimally dual handlebodies. We summarize this discussion in the following lemma, first observed in \cite{HT} and explicitly proven in Lemma 10 of \cite{Gay19}.

\begin{lemma}(\cite{HT}, \cite{Gay19})
Let $f_t$ for $t \in [0,1]$ be a generic path of Morse functions and let $c_1$ be a critical time. Take $\epsilon$ sufficiently small to have no other critical times in $[c_1-\epsilon, c_1+\epsilon]$. Then $H(f_{c - \epsilon})$ and $H(f_{c+\epsilon})$ are dual handlebodies. If $f_0 = f_1$ and $c_1, \cdots , c_k$ are the critical times for $f_t$, then $(H(f_{c_1 - \epsilon}), H(f_{c_2 - \epsilon}), \cdots, H(f_{c_k - \epsilon}))$ is a cyclic sequence of dual handlebodies.  
\end{lemma}

We next define an equivalence relation among paths of Morse functions which still captures the information of the handlebodies determined by the functions. Since only critical value exchanges in $S_{1,2}$ affected the handlebody, we are free to modify paths to include other types of singularities without changing the handlebodies defined. In addition, the particular path of Morse functions realizing a single crossing exchange in $S_{1,2}$ is immaterial in the handlebodies defined on each side, and we will also account for this.

\begin{definition}
\label{def:MorseRelations}

Let $f_t$ be a generic path of Morse functions, and let $c$ be a critical time for $f_t$ consisting of \textit{at most} one critical value exchange in a copy of $S_{1,2}.$ Let $[d,e]$ be a closed interval containing $c$ such that $f_t$ contains no other critical value exchanges in $S_{1,2}.$ We say that a generic loop of Morse functions, $f_t'$, is obtained by a  \textbf{torus crossing replacement} of $f_t$ if $f_t = f_t'$ outside of $[d, e]$ and $f_t'$ has  \textit{at most} one critical value exchange in a copy of $S_{1,2}$ in $[d, e]$.

We define the set $\mathbf{\mathfrak{M}_h(\Sigma_g)}$ to be the set of generic loops of Morse functions on $\Sigma_g$ up to torus crossing replacements and homotopies through generic loops of Morse functions. If $f$ is a Morse function, it may viewed as the constant loop in $\mathfrak{M}_h(\Sigma_g)$ and we denote its image in this set as $[f]_h$.
\end{definition}

\subsection{Pants and cut graphs}
Having defined our first object of interest, we next turn our attention to graphs associated to curves on a surface. In this section, we will show that loops in appropriate quotients of these complexes are in bijection with $\mathfrak{M}_h(\Sigma).$ We begin with the pants decompositions. The collection of all of the pants decompositions of a surface can be arranged into a graph described below, and whose edges can be seen in Figure \ref{fig:PantsComplexMoves}.
\begin{definition}
Given a closed, orientable genus $g \geq 2$ surface, $\Sigma_g$, the \textbf{pants graph} of $\Sigma_g$, denoted $P(\Sigma_g)$, is built from vertices and two types of edges as follows:
\begin{itemize}
    \item Vertices correspond to isotopy classes of pants decompositions of $\Sigma_g$.
    
    \item Two vertices have an \textbf {A-edge} between them if the corresponding pants decompositions share all but one curve,  and the curves which differ lie on a 4-punctured sphere and intersect twice. 
    
    \item Two vertices have a \textbf {S-edge} between them if the corresponding pants decompositions share all but one curve,  and the curves which differ lie on a punctured torus and intersect once. 

\end{itemize}

\end{definition}

\begin{figure}
    \centering
    \includegraphics[scale=.2]{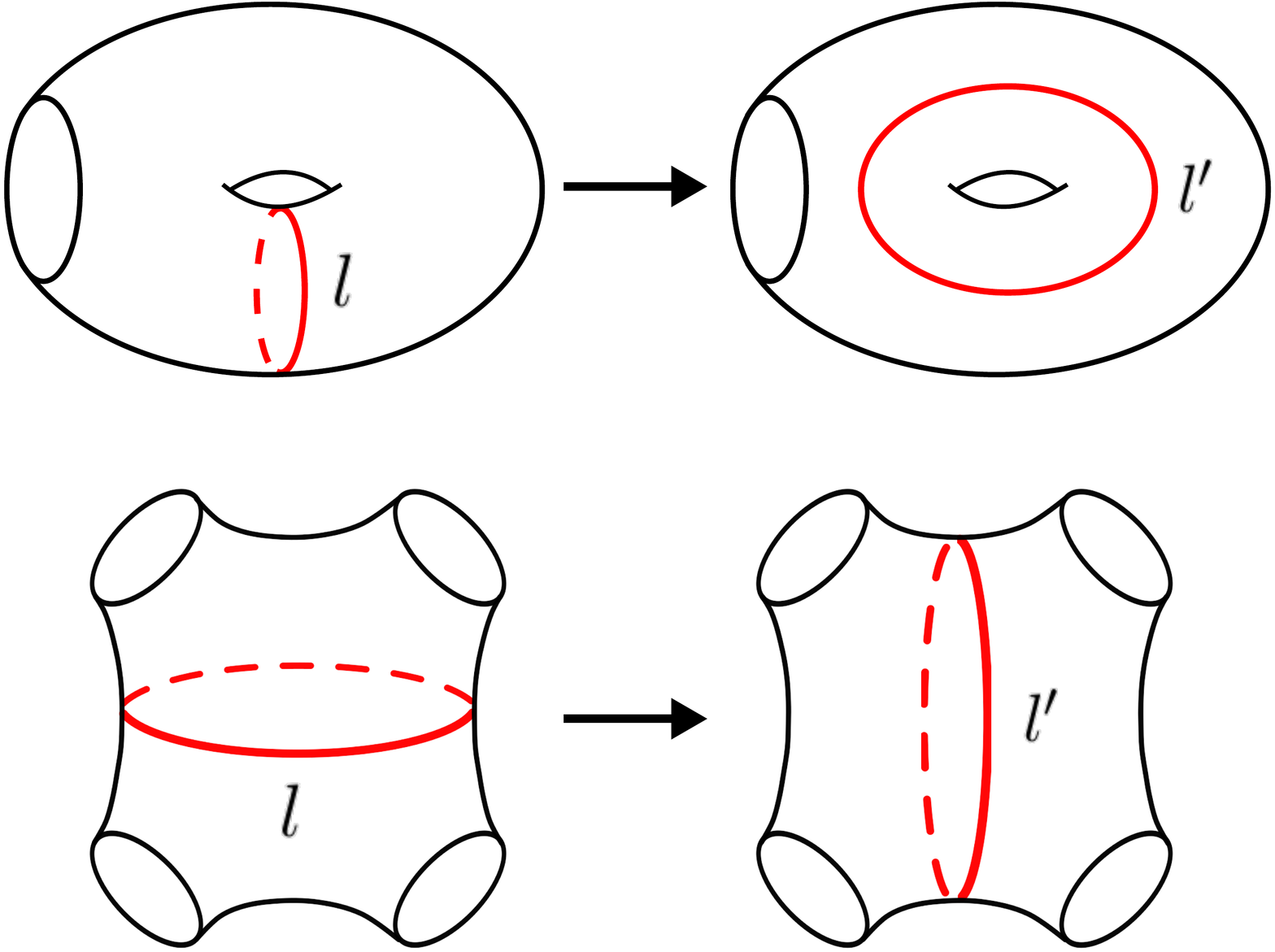}
    \caption{Top: An S-move in the pants complex. Bottom: and A-move in the pants complex}
    \label{fig:PantsComplexMoves}
\end{figure}

The other complex of interest in this paper will organize the cut systems of a surface into a complex. Following work of Wajnryb \cite{Waj98} and Johnson \cite{JJ}, Kirby and Thompson \cite{KT} arranged the cut systems of a given surface into a graph, whose edges can be seen in Figure \ref{fig:CutComplexEdges}, and which is defined below.

\begin{definition}
Given a closed, orientable genus $g$ surface, $\Sigma$, the \textbf{cut graph} of $\Sigma$, denoted $C(\Sigma)$, is built from vertices and two types of edges as follows:
\begin{itemize}
    \item Vertices correspond to isotopy classes of cut systems of $\Sigma$.
    
    \item Two vertices have a \textbf {Type-0} edge between them if the corresponding cut systems share all but one curve, and the curves which differ are disjoint. 
    
    \item Two vertices have a \textbf {Type-1} edge between them if the corresponding pants decompositions share all but one curve,  and the curves which differ intersect once. 
    \end{itemize}
\end{definition}

\begin{figure}
    \centering
    \includegraphics[scale=.4]{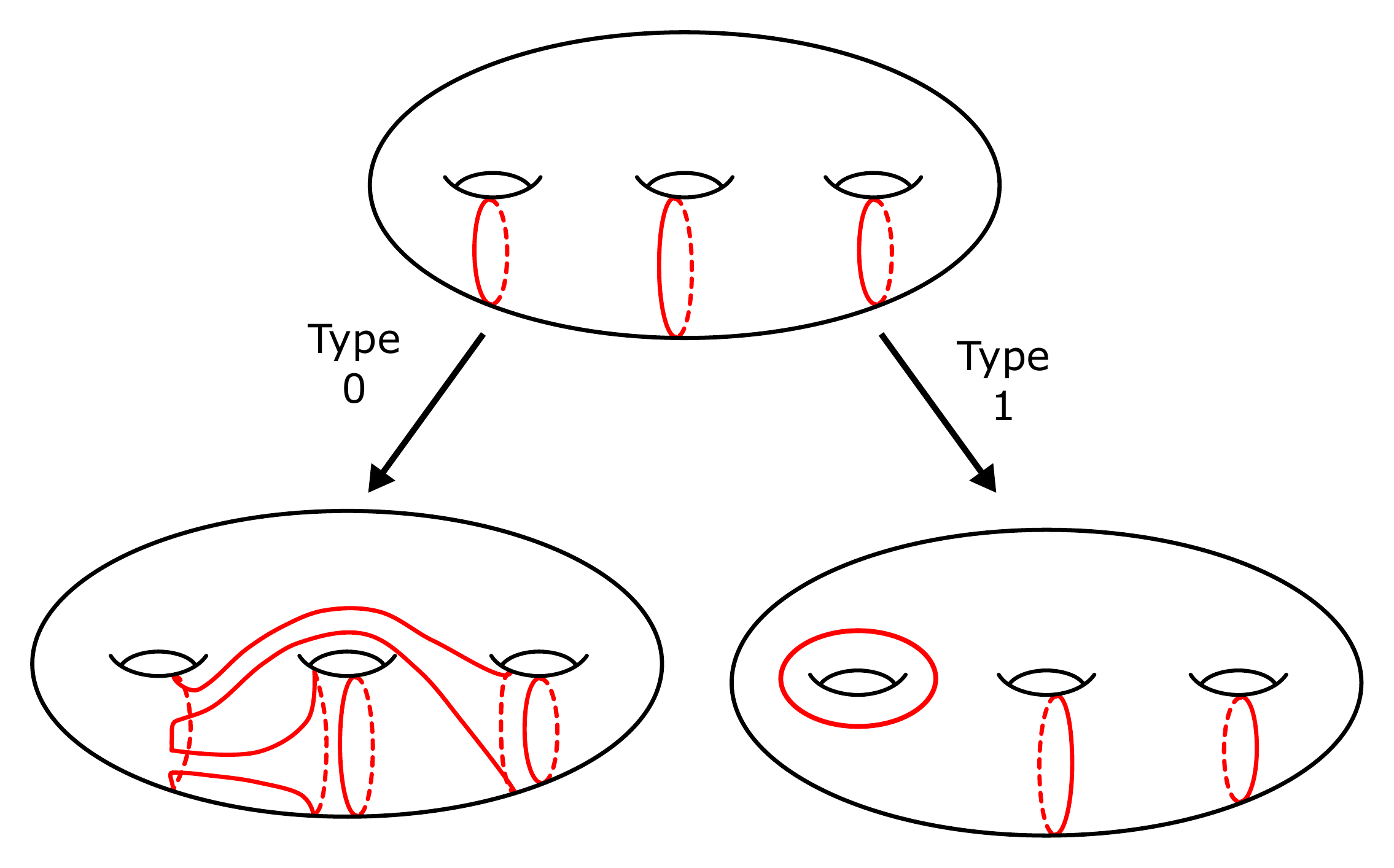}
    \caption{Edges in the Cut complex}
    \label{fig:CutComplexEdges}
\end{figure}

In the definition of the pants complex we excluded genus 1 surfaces. Here, following \cite{HatLocSch}, we define $P(\Sigma_1)$ to be $C(\Sigma_1)$. This complex only contains type 1 edges and is isomorphic to the Farey graph.

Slightly abusing notation, we will often refer to vertices in these complexes and the set of curves they correspond to on a surface using the same symbol. It is well known that a cut system determines a handlebody uniquely determined by the fact that all of the curves in the cut system bound a disk. Similarly, there is a unique handlebody corresponding to a pants decomposition, where all of the curves in the pants decomposition bound disks. Given a vertex $c \in C(\Sigma)$ or $p \in P(\Sigma)$ we denote this unique handlebody by $\textbf{H(c)}$ or $\textbf{H(p)}$, respectively. 

We define an equivalence relation on vertices in $C(\Sigma)$, $\sim_0$, to be $c \sim_0 c'$ if $c$ and $c'$ are connected by a sequence of only type-0 edges. Similarly, if $p,p' \in P(\Sigma),$ then we say $p \sim_A p'$ if $p$ and $p'$ are connected by a sequence of only A-edges. After taking the initial quotients by these relations, two vertices may have multiple edges between them. We will further quotient any two edges relating the same vertex to obtain the complexes $\mathbf{C(\Sigma) / \sim_0}$ and $\mathbf{P(\Sigma) / \sim_A}$. We refer to the images of vertices $c \in C(\Sigma)$ and $p \in P(\Sigma)$ in the quotient complexes $C(\Sigma) / \sim_0$ and $P(\Sigma) / \sim_A$  by $[c]_0$ and $[p]_A$ respectively. 

\begin{lemma}
The quotient graphs $C(\Sigma_g) / \sim_0$ and $P(\Sigma_g) / \sim_A$ are isomorphic.
\end{lemma}

\begin{proof}
We first show that vertices in the quotient complexes correspond to handlebodies. Since any two cut systems determining a handlebody are related by handle slides and since handle slides are a subset of type-0 moves, any two vertices representing the same handlebody are connected by a path of type-0 moves. Similarly, Luo \cite{FeLu} showed that any two vertices in the pants complex are connected by a path of $A$-moves. Moreover if $c \sim_0 c'$ or $p \sim_A p'$ then the handlebodies corresponding to the two vertices are the same. Together these facts imply that the vertices of both complexes correspond to handlebodies with boundary $\Sigma_g$. 

To fix an explicit map, $\phi: C(\Sigma_g) / \sim_0 \to P(\Sigma_g) / \sim_A$ between the vertices of the complexes, we can start with an arbitrary vertex $[c]_0 \in C(\Sigma_g) / \sim_0$, take a representative $c \in C(\Sigma_g)$ of $[c]_0$ and extend $c$ to a pants decomposition, $p$, by adding $2g-3$ curves in an arbitrary fashion. Once a cut system is specified, insisting that more curves in the complement of that cut system bound a disk in a handlebody is superfluous information, so that $H(c) = H(p)$ regardless of the choices we made in this extension. This means that setting $\phi([c]_0) = [p]_A$ is a well defined bijection on the level of vertices.

We next seek to extend $\phi$ to edges. If two vertices $[c]_0$ and $[c']_0$ share an edge in $C(\Sigma_g) / \sim_0$ then there exist vertices $c$ and $c'$ in $C(\Sigma_g)$ which are related by a type-1 edge. Then there are two curves in the cut systems $c$ and $c'$ which intersect once. Taking a neighbourhood of these curves gives a separating curve, $s$, which separates a punctured torus containing the curves. Extending the cut systems identically to pants decompositions, $p$ and $p'$ containing $s$, we see that there is a type-s edge between $p$ and $p'$. So that there is an edge between $[p]_A = \phi([c]_0)$ and $[p']_A = \phi([c']_0)$. The converse, that $\phi([c]_0)$ and $\phi([c']_0)$ share an edge only if $[c]_0$ and $[c']_0$ do, follows by a similar analysis, this time by removing curves from a pants decomposition to obtain cut systems. Therefore, $\phi$ extends to an isomorphism of graphs.
\end{proof}

In the proof above we saw that $\sim_0$ and $\sim_A$ simply relate vertices corresponding to the same handlebody, so that $H([c]_0)$ and $H([p]_A)$ are well defined handlebodies. Moreover, two handlebodies corresponding to adjacent handlebodies in these complexes form a Heegaard splitting for $\#^{g-1} S^1 \times S^2$ so that they are, in particular, dual handlebodies. We therefore call the quotient graph $C(\Sigma) / \sim_0 \, \cong P(\Sigma) / \sim_A$ the \textbf{dual handlebody graph} and denote it $\mathbf{D(\Sigma)}$. Given the preceding discussion, the following is immediate.

\begin{lemma}
Let $D = (d_1, d_2, ... , d_n)$ be a loop in the dual handlebody graph. Then $(H(d_1), H(d_2), ... , H(d_n))$ is a cyclic sequence of dual handlebodies.
\end{lemma}

\section{First Bijections}
\label{sec:firstBijections}
\subsection{Morse functions and the dual handlebody complex}

In this subsection we seek to produce a bijection, $D$, between  $\mathfrak{M}_h(\Sigma)$ and loops in $D(\Sigma)$. We define the map below, noting that we will later prove well definedness. 

\begin{definition}
Let $f_t$ be a generic loop of Morse functions, and let $[f_t]_h$ be its image in $\mathfrak{M}_h(\Sigma)$. Let $c_1...c_n$ be the critical times of $f_t$ corresponding to a critical value exchange in a copy of $S_{1,2}$. Let $r_1, ... r_n$ be regular times for $f_t$ with $r_1<c_1<r_2<c_2<...<r_n<c_n$. The map $$D: \mathfrak{M}_h(\Sigma) \to \text{Loops in } D(\Sigma) $$ is given by $D([f_t]_h) = (H(f_{r_1}),  H(f_{r_2}),....  H(f_{r_n}))$
\end{definition}

Much of the work in proving $D$ is a surjection is carried out in \cite{HatLocSch} where the authors assign a loop in $P(\Sigma)$ to a generic loop of Morse functions on $\Sigma$. In that work, there is not a canonical choice of loop in $P(\Sigma)$ corresponding to a loop of generic Morse functions, however by attaching 2-cells to $P(\Sigma)$, the choice can be made unique up to homotopy in the resulting 2-complex. In this paper, the quotients defined deal with the ambiguity.

\begin{proposition}
The map $D$ is a well defined surjection.
\end{proposition}

\begin{proof}
We will work with the model of $D(\Sigma)$ given by loops in the pants complex modulo A-moves. In Lemma 3 of \cite{HatLocSch}, it is shown that the assignment of a Morse function, $f$, on $\Sigma$ to $P(f)$ is surjective on the level of vertices. Moreover, given a generic path of Morse functions $f_t$ which is critical at $t=c$, they analyze the change between $P(f_{c-\epsilon})$ and $P(f_{c+\epsilon})$ to assign a path, $P(f_t)$, in $P(\Sigma)$ to $f_t$. Finally, they show that every path is $P(f_t)$ for some generic path of Morse functions $f_t$.

The path $P(f_t)$ in $P(\Sigma)$ is not canonical, and there are two situations in which we must make choices. If $c$ is a critical time involving a specific configuration in a 4-holed sphere then, up to a homeomorphism of $\Sigma$, $P(f_t)$ changes from the cut system on the left of Figure \ref{fig:4HSAmbiguity} to the cut system on the right of Figure \ref{fig:4HSAmbiguity}. There are also two equally valid paths of minimal length between these pants decompositions shown in the figure. Since all of the moves involved here are A-moves, all of these pants decompositions are actually the same point in $D(\Sigma)$ so that the ambiguity disappears in the quotient. 

In the second case, $c$ is a critical time involving a crossing change in a 2-holed torus. Here, up to a homeomorphism of $\Sigma$, $P(f_t)$ changes from the cut system on the left of Figure \ref{fig:6ASPath} to the cut system on the right of Figure \ref{fig:6ASPath}. Also shown in this figure are two equally valid paths in $P(\Sigma)$ between these two cut systems. Note, however, that after quotienting by $A$ moves, both of these paths collapse to a unique edge between the left and right cut systems. Therefore, we have a unique path $D(f_t)$ in $D(\Sigma)$ corresponding to a generic path $f_t$.
\end{proof}

\begin{figure}
    \centering
    \includegraphics[scale=.2]{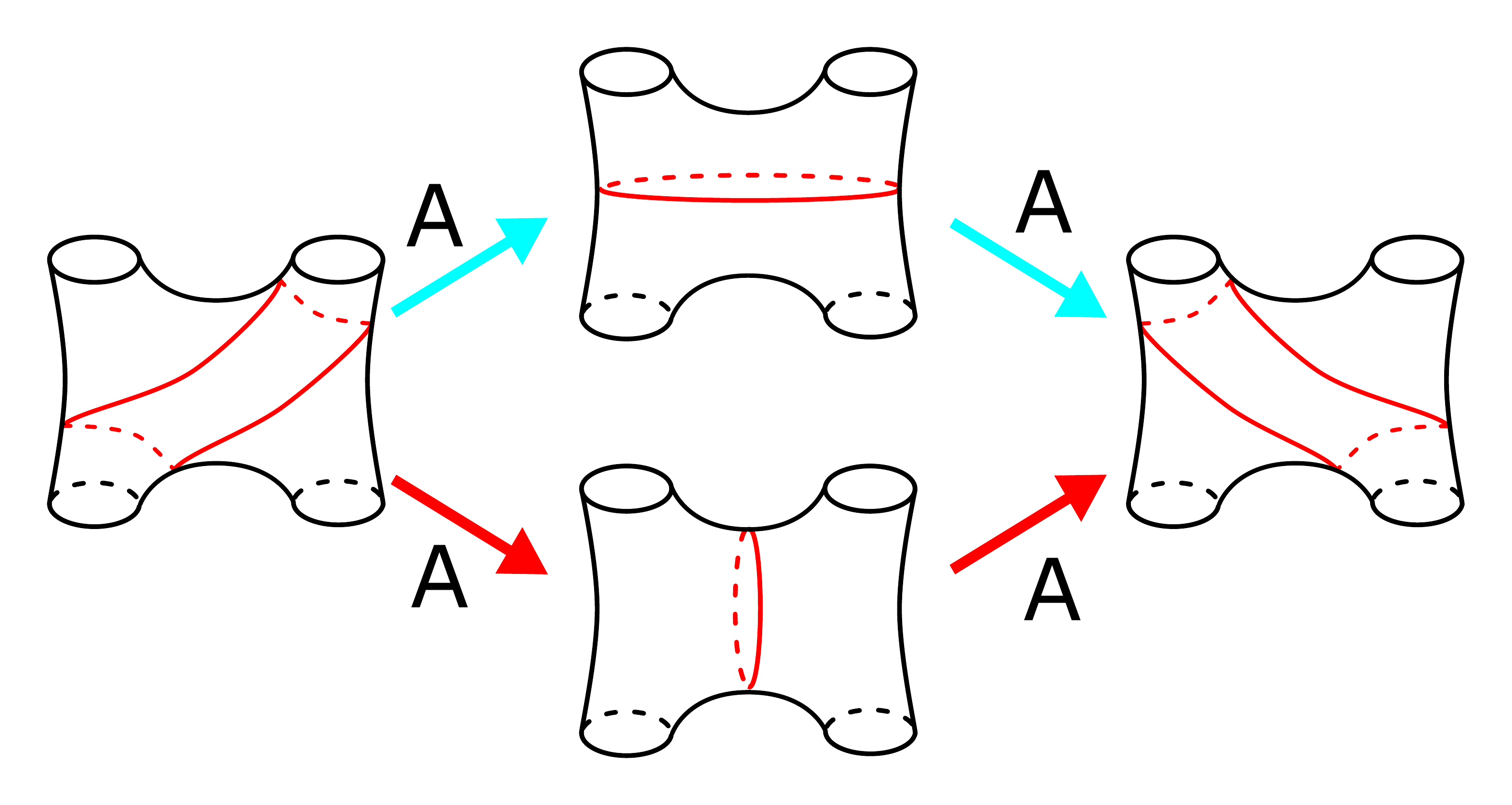}
    \caption{Two paths in the pants complex which both become a constant path when quotienting by A-moves to pass to the dual handlebody complex.}
    \label{fig:4HSAmbiguity}
\end{figure}

\begin{figure}
    \centering
    \includegraphics[scale=.2]{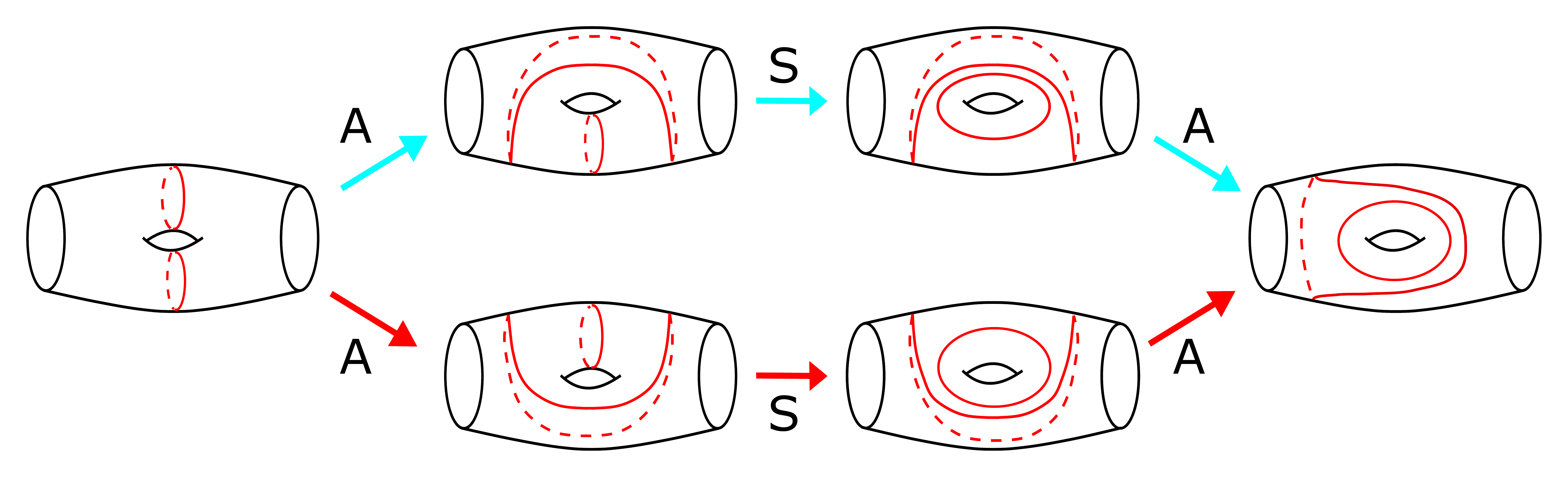}
    \caption{Two paths in the pants complex between the same vertices. These two paths collapse to the same edge in the dual handlebody complex.}
    \label{fig:6ASPath}
\end{figure}

We next seek to show that the map $D$ defined above is an injection. The following theorem, which the author of \cite{Kud99} attributes to Matveev will be useful in our proof.

\begin{theorem} (Theorem 8 of \cite{Kud99})
\label{thm:boundaryMorseConnected}
Let $S_{g,b}$ be a surface of genus $g$ with $b$ boundary components split into two collections, $b_+$ and $b_-$. Let $f_0: S_{g,b} \to \mathbb{R}$ and  $f_1: S_{g,b} \to \mathbb{R}$ be two Morse functions which take the constant value of $-1$ on the circles in $b_-$ and the constant value of $1$ on the circles in $b_+$. Then $f_0$ and $f_1$ are connected by a generic path, $f_t$, which is constant in a neighbourhood all of the boundary circles. 

\end{theorem}

Recall that in Definition \ref{def:MorseRelations}, we defined a relation on paths of Morse functions which we called a torus crossing replacement. This relation involved paths of Morse functions which involve at most one critical value crossing in a copy of $S_{1,2}$. The following proposition will give us a useful path of Morse functions involving no critical value crossings in  $S_{1,2}$ between Morse functions on $\Sigma_g$ yielding the same handlebody.

\begin{proposition}
\label{prop:MorseFunctionPath}
Let $f_0: \Sigma_g \to \mathbb{R}$ and $f_2: \Sigma_g \to \mathbb{R}$ be Morse functions with $H(f_0) = H(f_2)$. Then there exists a generic path of Morse functions, $f_t$, such that $f_t$ contains no critical value crossings in a copy of $S_{1,2}$.
\end{proposition}

\begin{proof}
We will show that both functions can be connected by such a path to an intermediary function $f_1$. Consider the Morse function $f_1: \Sigma \to \mathbb{R}$ given by the height function of the surface in Figure \ref{fig:f1}. Pictured is the genus 3 case but the important property of this function is that all splitting saddles occur before any merging saddles. After an overall homeomorphism of $\Sigma$, we may assume that the handlebody $H(f_0)=H(f_2)$ is the same as $H(f_1)$. Since the handlebodies $H(f_0)$ and $H(f_1)$ are the same, Corollary 1 of \cite{FeLu} implies that the pants decompositions $P(f_0)$ and $P(f_1)$ are related by a sequence of A-moves. By Lemma 3 of \cite{HatLocSch},  each of these A-moves can be realized as a crossing change of the Morse function in a 4-punctured sphere. After all of these A-moves have been realized, we have a homotopy $f_t$, $t\in [0, \frac{1}{6}]$, with $P(f_{\frac{1}{6}}) = P(f_1)$.

We next seek to make the values of $f_{\frac{1}{6}}$ and $f_1$ identical on a cut system. We consider the cut system, $C = C(f_1)$, shown in Figure \ref{fig:f1}. These are components of level sets obtained immediately after each saddle singularity which splits a level set component into two components. Since all of the splitting saddles occur before any of the merging saddles, we can freely interchange the relative order of the cut system given by the height. After arranging the curves in the cut system of $f_1$ so that they occur in the same order as that of $f_{\frac{1}{6}}$, we may dilate or contract regions of $f_1$ until the curves in $C$ take the same values as they do in $f_{\frac{1}{6}}$. We let (the reverse of) this homotopy through Morse functions be $f_t$ for $t \in [\frac{5}{6}, 1]$.

Next, we seek to have $f_{\frac{1}{6}}$ and $f_{\frac{5}{6}}$ agree in a neighbourhood of $C$. The only ambiguity remaining in the neighbourhood of a curve is the normal direction to the curve given by increasing values of the Morse function. This can be changed by locally modifying the function around each curve as shown in Figure \ref{fig:reverseNormal}. Using this local modification, we change the normal directions to the curves in $C$ in $f_{\frac{1}{6}}$ to agree with those in $f_{\frac{5}{6}}$ by a homotopy $f_t$ $t \in [\frac{1}{6}, \frac{2}{6}]$ with all critical times corresponding to birth/death singularities.

The neighbourhood of each curve in $C$ has two boundary components, one of which takes values lower than the curve, and the other of which takes larger than the curve. Remove a neighbourhood of $C$ from $\Sigma$ to obtain a sphere with $2g$ boundary components. In both functions, take the resulting lower value boundary component and drag it above the rest of the surface and also drag the higher value boundary component below the rest of the surface. We now have two relative Morse functions  $f_{\frac{2}{6}}'$ and  $f_{\frac{5}{6}}'$ on a $2g$ punctured sphere built from $g$ circles. Note that any critical value exchanges in these functions can not occur in a $S_{1,2}$, as the resulting surface has no subsurfaces homeomorphic to $S_{1,2}$.

Cancel any maxima and minima in the resulting functions by homotopies $f_t'$ for $t \in [\frac{2}{6}, \frac{3}{6}]$, and $t \in [\frac{4}{6}, \frac{5}{6}]$. Then $f_\frac{3}{6}'$ and $f_\frac{4}{6}'$ are relative Morse functions with $g$ minima circles and $g$ maxima circles. By Theorem \ref{thm:boundaryMorseConnected}, the space of such functions is connected through a generic path $f_t$ for  $t \in [\frac{3}{6}, \frac{4}{6}]$ which is constant near the boundary circles. Since the path is constant near the boundary circles it can be viewed as a generic path of functions on the original surface $\Sigma$, giving the desired homotopy $f_t$ for $t \in [\frac{2}{6}, \frac{5}{6}]$ thus completing the homotopy $f_t$ for $t \in [0,1].$ Repeating the process we can find a generic path $f_t$ for $t \in [1,2]$. Composing the two homotopies gives the desired result.
\end{proof}

\begin{figure}
    \centering
    \includegraphics[scale=.35]{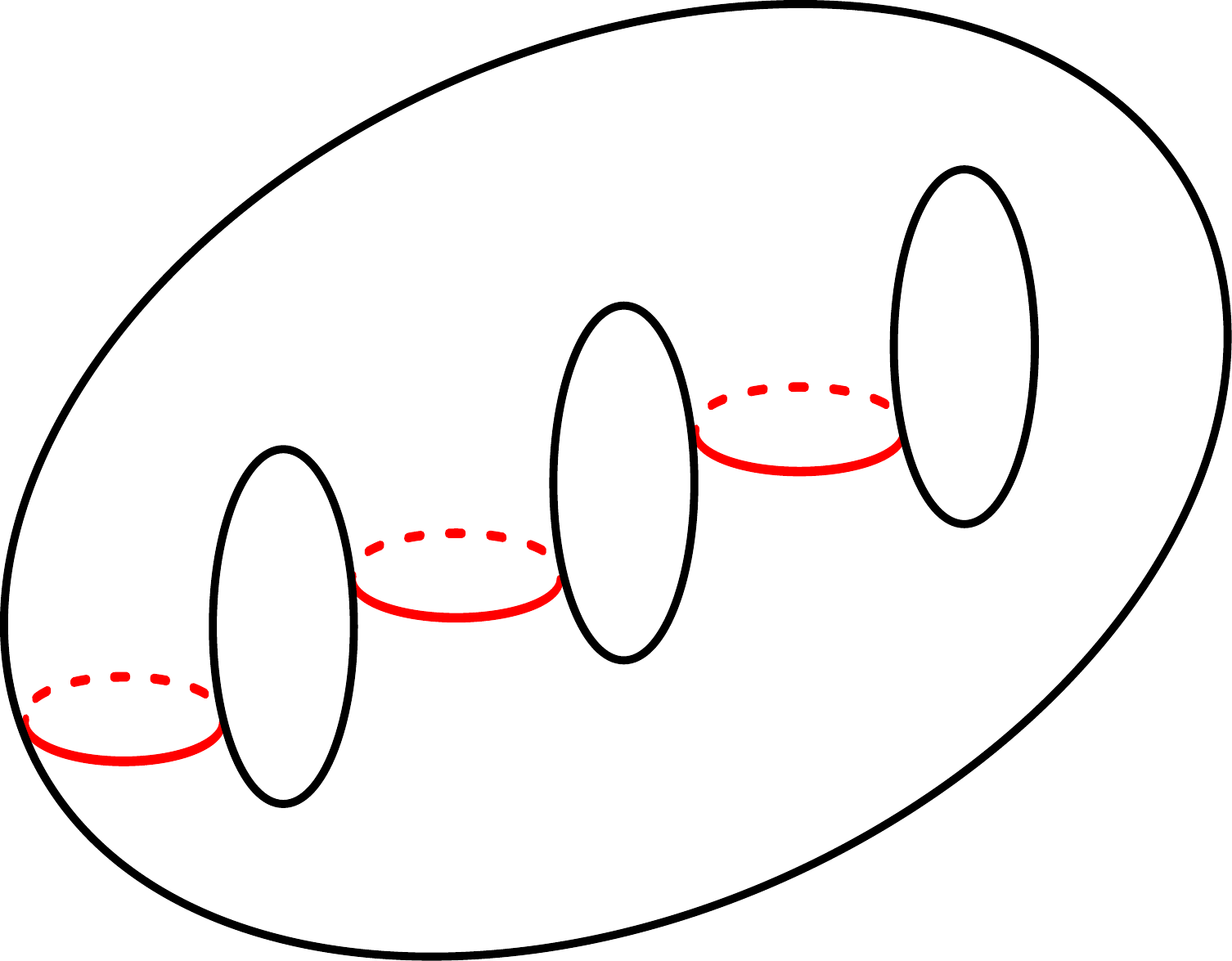}
    \caption{The Morse function $f_1$ in the proof of Proposition \ref{prop:MorseToComplexInjection} is the height function of the surface pictured. The level curves shown can be homotoped to lie in any order with respect to the height function.}
    \label{fig:f1}
\end{figure}

\begin{figure}
    \centering
    \includegraphics[scale=.4]{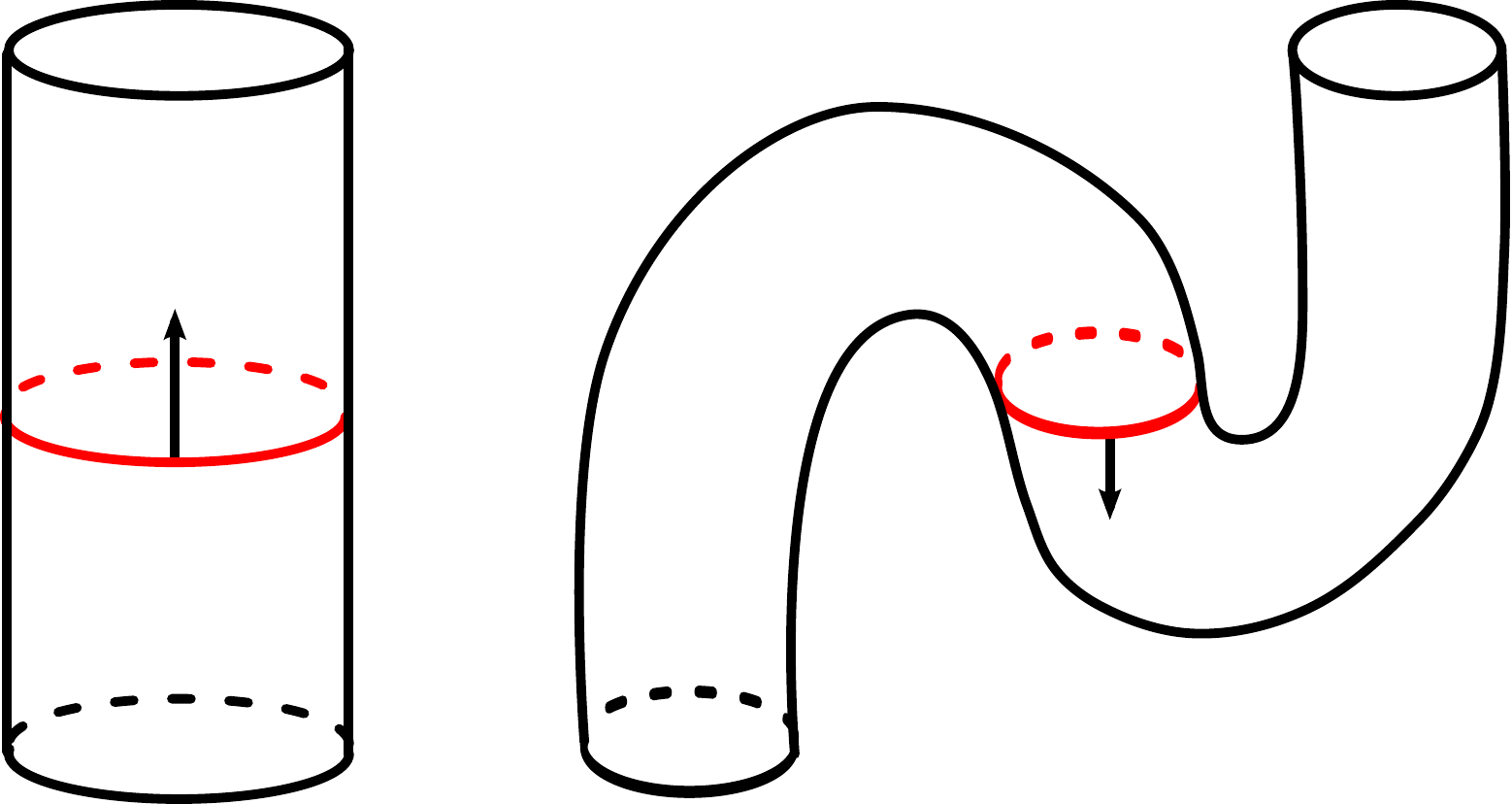}
    \caption{We can invert the normal to a curve given by increasing values of a Morse function at the cost of introducing two cancelling pairs.}
    \label{fig:reverseNormal}
\end{figure}

Given a generic path of Morse functions $f_t$, we may draw the critical value set in $\mathbb{R} \times I$ as a Cerf graphic \cite{JC}. To add clarity to this picture, we will place all of singularities of $f_t$ which are not critical value crossings in a torus into a Cerf box and draw all of the tori crossings to highlight these critical times. Examples of such graphics with two critical value crossings in tori are given in Figure \ref{fig:cerfBoxReplacement}. Proposition \ref{prop:MorseFunctionPath} showed that given two Morse functions $f_0$ and $g_0$ with $H(f_0) = H(g_0)$ we can find a path with no critical value crossings in tori. We represent this homotopy by the Cerf box $B_{fg_0}$ and it is the main input into the main proposition of this section. The proof is also illustrated in Figure \ref{fig:cerfBoxReplacement}.

\begin{proposition}
\label{prop:MorseToComplexInjection}
The assignment $D: \mathfrak{M}_h(\Sigma_g) \to \text{Loops in } D(\Sigma_g)$ is an injection.
\end{proposition}

\begin{proof}
Let $f_t$ and $g_t$ be loops of Morse functions with $D(f_t) = D(g_t)$. As the images of the loops have the same length in $D(\Sigma_g)$ they both must contain the same number of critical value crossings in tori, and we call this number $n$. By a homotopy through generic paths we may assume that $t \in [0,n]$ and that the functions $f_0$, $f_1$,... $f_{n-1}$ and  $g_0$, $g_1$,... $g_{n-1}$ are all regular times for the paths $f_t$ and $g_t$ respectively and that there is exactly one critical value crossing in a torus between these times. Since $D(f_t) = D(g_t)$ we may assume, after a cyclic reordering, that $H(f_i) = H(g_i)$. 

By Proposition \ref{prop:MorseFunctionPath} there exists a path $fg^i_t$ of Morse functions with $fg^i_0 = f_i$ and  $fg^i_1 = g_i$ with no critical values crossings in tori. After a homotopy through generic paths, we may assume that $f_t$ is constant in $t$ for a short amount of time after $f_i$, for all $i \in \{1 , ..., n-1 \}$. At each $i$, we may then replace the constant homotopy by the homotopy $fg^i_t$ followed by the inverse homotopy $fg^i_{1-t}$ at every $i$. Since these homotopies have no crossings in tori, we have performed a (trivial) torus crossing replacement. The Cerf graphic at this stage is shown in the middle of Figure \ref{fig:cerfBoxReplacement}.

Now in between the inserted homotopy $fg^i_t$ and its inverse we have the function $g_i$. Moreover, in between each $g_i$ and $g_{i+1}$ (with indices taken mod $n$) there is at most critical value exchange in a torus. We may then replace this portion of the homotopy with $g_t$, $t \in [i,i+1]$ by a torus crossing replacement. At the end of this process we have reached the generic path $g_t$ $t \in [0,n]$ so that $f_t$ and $g_t$ are related by torus crossing replacements and homotopies through generic paths, as desired. 

\end{proof}

\begin{figure}
    \centering
    \includegraphics[scale=.28]{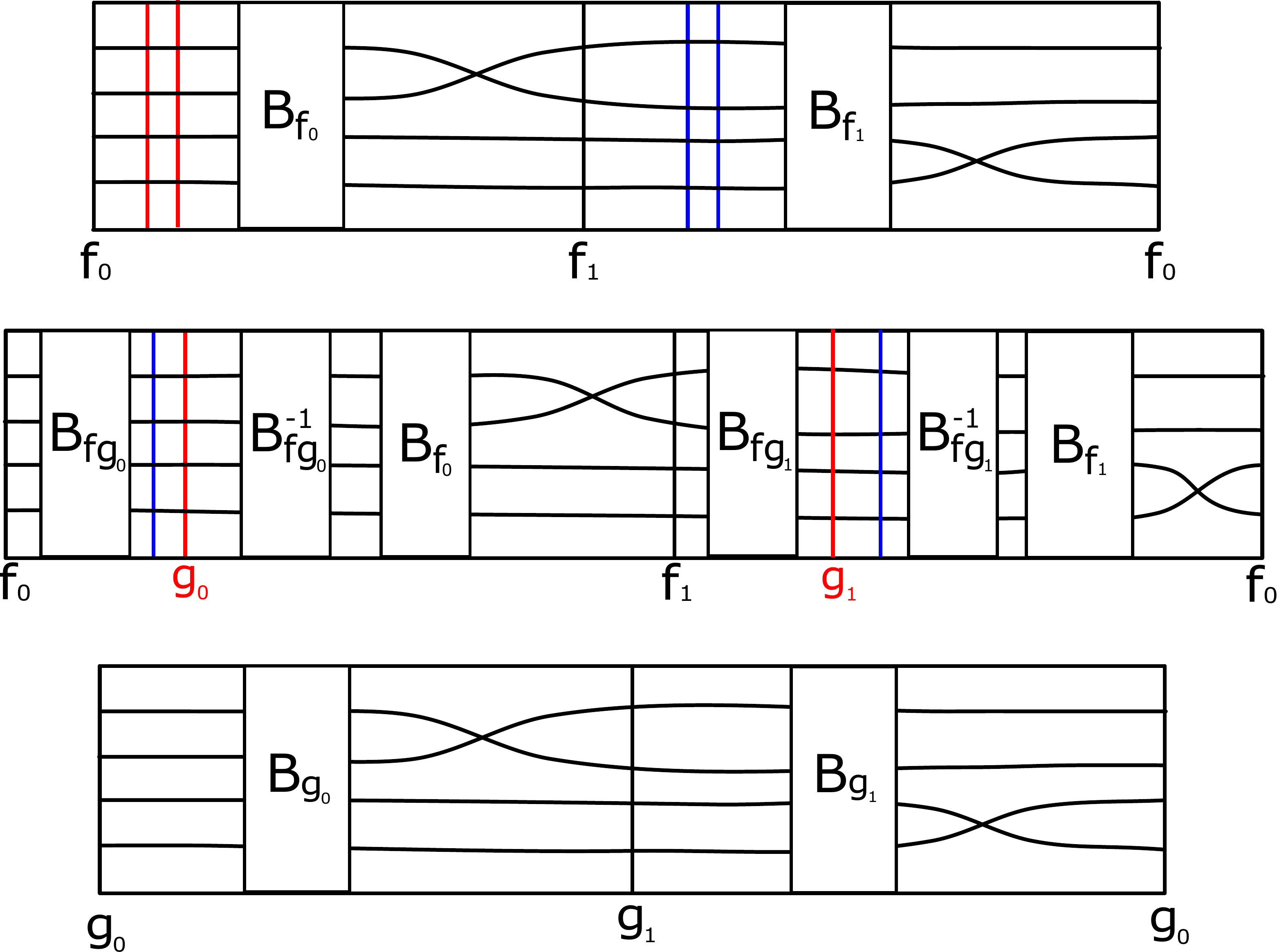}
    \caption{An example of the torus crossing replacements in Proposition \ref{prop:MorseToComplexInjection} in the case of two crossings in tori. In going from the top picture to the middle picture the constant homotopies in the red and blue boxes are replaced by the homotopies between $f_0$ and $g_0$ and $f_1$ and $g_1$, composed with their inverses. In going from the middle picture to the bottom picture, the blue box wraps around from the right side of the Cerf graphic to the left side of the Cerf graphic. }
    \label{fig:cerfBoxReplacement}
\end{figure}

\subsection{The dual handlebody complex and thin multisections}
\label{sec:DHtoMS}

In Section \ref{sec:constructingMultisections} we described how to take a sequence of dual handlebodies and produce a  multisection. We also noted that if the handlebodies were minimally dual, then the induced multisection was thin. By factoring through this construction, we get a map, $T$, from loops in $D(\Sigma)$ to the set of thin multisections with multisection surface $\Sigma_g$ up to diffeomorphism, which we denoted $T(\Sigma_g)$. We seek to show that $T: \text{loops in } D(\Sigma_g) \to T(\Sigma_g)$ is a bijection.

\begin{proposition}
\label{prop:loopsAndThinMultisections}
$T$ is a bijection between $\text{loops in } D(\Sigma_g)$ and the set $T(\Sigma_g)$ of thin multisections with multisection surface $\Sigma_g$ up to diffeomorphism.
\end{proposition}

\begin{proof}
We first show that $T$ is injective. Let $L = (H_1, H_2, ..., H_n)$ and let $L' = (H_1', H_2', ..., H_n')$ be two loops in $D(\Sigma_g$ with $T(L) = T(L')$. Then the multisections $X^4(L) = X_1 \cup X_2 \cup ... \cup X_n$ and $X^4(L') = X_1' \cup X_2' \cup ... \cup X_n'$ are diffeomorphic as multisections. Recall that this means that there is a diffeomorphism of the 4-manifolds rel $\Sigma_g$ sending $X_i$ to $X_i'$. In the construction of $T(L)$ we have set $\partial X_i = H_i \cup H_{i+1}$.  Restricting the overall diffeomorphism of the multisection to each $H_i$ gives a homeomorphism rel $\Sigma_g$ from $H_i$ to $H_i'$ so that $H_i = H_i'$ and so $L = L'$.

Next we show that $T$ is surjective. Let $M(X)$ be a thin multisection of $X$ given by $X = X_1 \cup X_2 \cup .. \cup X_n$. Then (with notation as in Definition \ref{def:multisection}) the handlebodies $H_{i-1,i}$ and $H_{i,i+1}$ are minimally dual handlebodies. Then the sequence $(H_{1,2}, H_{2,3}, ... H_{n,1})$ corresponds to a loop $L$ in $D(\Sigma)$ with $T(L) = M(X)$ up to possibly a diffeomorphism of the multisection $M(X).$
\end{proof}

We compile the results of this section into the following summary theorem.

\begin{theorem}
\label{thm:firstBijections}
There are bijections $\mathfrak{M}_h(\Sigma_g) \leftrightarrow \text{Loops in } D(\Sigma_g)  \leftrightarrow T(\Sigma_g)$.
\end{theorem}

\section{Morse 2-functions and wrinkled fibrations}

In order to prove our correspondence between the sets in Theorem \ref{thm:firstBijections} and smooth 4-manifolds, we will pass through the theory of stable maps and wrinkled fibrations. Each thin multisection gives rise to a smooth map from the 4-manifold to the disk. We will show how to modify these maps in order to produce new multisections, and we will show that these modifications suffice to relate any two thin multisections of the same 4-manifold.

\begin{definition}
A Morse 2-function is a map $f: M^4 \to D^2$ which is smooth, stable, and generic among smooth $D^2$ valued functions. 
\end{definition}

In this paper, we will be particularly concerned with Morse 2-functions which give rise to  multisections. The following class of Morse 2-functions, which generalize the definition of a trisected Morse 2-function, are such a class of functions..\

\begin{definition}
\label{def:MultisectedMorse2}
A Morse 2-function $F:X^4 \to D^2$ is said to be \textbf{radially monotonic} if
\begin{enumerate}
    \item $F$ has a unique definite fold on the outside of the disk;
    \item Indefinite folds always have index 1 when moving towards the center of the disk;
    \item There exist $n$ radial lines separating the disk into $n$ sectors, such that each fold has at most one cusp in each sector.
\end{enumerate}

We call such a Morse 2-function, together with a collection of radial lines satisfying the last condition a \textbf{multisected Morse-2 function}. We call the inverse image of the center of the disk the \textbf{multisection surface} and will often fix an identification of this surface with our model surface $\Sigma_g$.
\end{definition}

By decomposing the 4-manifold into the pieces lying in the inverse image of each sector of the disk, a multisected Morse 2-function gives rise to a unique multisection on $X^4$. We call this multisection $\textbf{M(F)}$. Starting at the center of the disk at $\Sigma_g$ and following a radial line $r$ outwards, we pass through exactly $g$ folds of index 2 and a single fold of index 0. Thus, the manifold $F^{-1}(r)$ is a handlebody with boundary $\Sigma_g$. We call this handlebody $\mathbf{H(r)}$. If $(r_1, ... r_n)$ is the ordered collection of radial lines, then by the construction in Section \ref{sec:4mfldFromHB} the multisection structure can be recovered from the cyclic sequence of dual handlebodies $(H(r_1), H(r_2), ... , H(r_n))$ with boundary $\Sigma_g$.

A multisected Morse 2-function $F$ is said to be \textbf{thin} if each sector contains exactly 1 cusp. If a multisected Morse 2-function $F$ is thin, then the induced multisection $M(F)$ is also thin. Given a multisected Morse 2-function which is not thin, some sector will contain multiple cusps. We can create a new multisected Morse 2-function by separating these cusps into multiple sectors with an additional radial line. We can also do the opposite and put multiple cusps into the same sector by removing a line, though this is not always possible while maintaining the one cusp per fold condition of Definition \ref{def:MultisectedMorse2}. We define these operations below.

\begin{definition}
A multisected Morse 2-function $F'$ is said to be an \textbf{expansion} of $F$ if $F'$ is obtained by adding a collection of radial lines to $F$. Here we say that $F$ is a \textbf{contraction} of $F'$. 
\end{definition}

When attempting to expand a thin multisection, one simply produces regions with no cusps in them. These correspond to product regions of the manifold and contain only redundant information. In particular, these regions can be contracted into adjacent sectors without changing their topology. Due to this, we will not consider multisection maps with regions with no cusps, nor multisections which contain product regions as a sector.

In the course of modifying Morse 2-functions, we will pass through some functions which contain additional singularities, which are not smoothly stable. These are the wrinkled fibrations of \cite{Lek09}.

\begin{definition}
 A \textbf{wrinkled fibration} $W: M^4 \to D^2$ is a function which is a Morse 2-function everywhere except at a finite number of isolated points where the function has a Lefschetz singularity.
\end{definition}

\subsection{Reference paths and modifications of wrinkled fibrations}

 We will be concerned with how to import information of the singular sets of our maps onto a reference surface. This will be facilitated by the following definition.

\begin{definition}
Let $f: M^4 \to D^2$ be a wrinkled fibration, $b \subset D^2$ be a regular value, and $c$ be a critical value. A \textbf{reference path} is an arc in $D^2$ with endpoints $b$ and $c$ such that when traversing the arc from $b$ to $c$ only index-2 folds are crossed. Reference paths are considered up to homotopies supported away from cusps and Lefschetz singularities.
\end{definition}

In the definition above, the point $c$ may be a Lefschetz singularity, so long as no interior points of the arc pass through Lefschetz singularities or cusps. Given a directed reference path $p$, starting at $b$, we can draw the vanishing cycle  of a Lefschetz singularity or an index-2 fold singularity of $c$ on $f^{-1}(b) := \Sigma$ by parallel transport.  We call the resulting curve $c_{p}$. Intuitively, $c_{p}$ is the curve which contracts as we approach $c$ via the path $p$.  See \cite{BehHay16} for details on this construction.

We next seek to understand how vanishing cycles change as we modify a Morse 2-function. Following \cite{BS}, we say a local modification of critical images which takes a critical image $C_0$ and produces a critical image $C_1$ is \emph{always realizable} if there is a smooth 1-parameter family of smooth maps $F_t: X^4 \to D^2$ such that the critical image of $F_0$ is $C_0$ and the critical image of $F_1$ is $C_1$. We will need three particular always realizable moves in this paper. 

The first is the \emph{unsink} move of \cite{Lek09}, which takes folds of index 1 and index 2 meeting at a cusp, and transforms them into a fold with no cusp, together with a Lefschetz singularity. The modification of critical images, together with the resulting vanishing cycles along the reference paths drawn is given in Figure \ref{fig:unsink}. In particular if the vanishing cycle associated to first indefinite fold (ordered by increasing counterclockwise angle) is $\alpha$ and the second is $\beta$, then the resulting Lefschetz singularity has vanishing cycle $\tau_{\alpha}(\beta)$, where $\tau$ denotes a right handed Dehn twist.

\begin{figure}
    \centering
    \includegraphics[scale=.3]{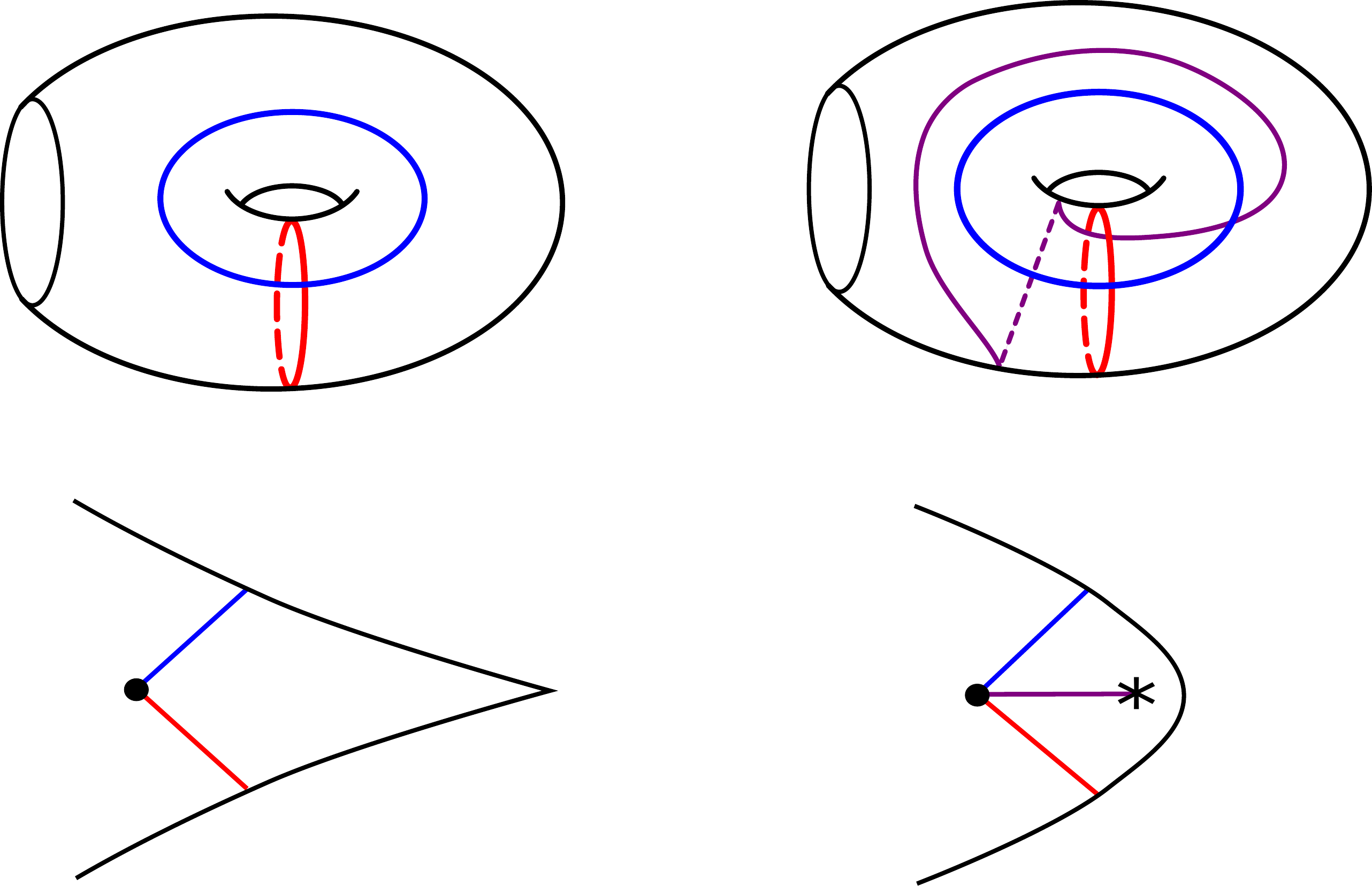}
    \caption{The resulting critical images and vanishing cycles when performing an unsink move.}
    \label{fig:unsink}
\end{figure}

The second is the \emph{push} move, which moves a Lefschetz singularity over an indefinite fold towards the side with the higher genus fiber. This modifies the critical image in a straightforward way and the vanishing cycle associated to the Lefschetz singularity is unchanged. This move was introduced by Baykur in \cite{Bay09}.

The third is the \emph{wrinkle} move of \cite{Lek09}, which transforms a Lefschetz singularity into a fold with three cusps. The modification of critical images, together with the resulting vanishing cycles along the reference paths drawn is given in Figure \ref{fig:wrinkle}. The three resulting vanishing cycles of the indefinite folds are characterized by the property that they each intersect transversely once positively and that any arc from one boundary component of the surface to the other must intersect one of the curves. We note that the effect of wrinkling on relative trisection diagrams is studied in \cite{CasOzg}.

\begin{figure}
    \centering
    \includegraphics[scale=.3]{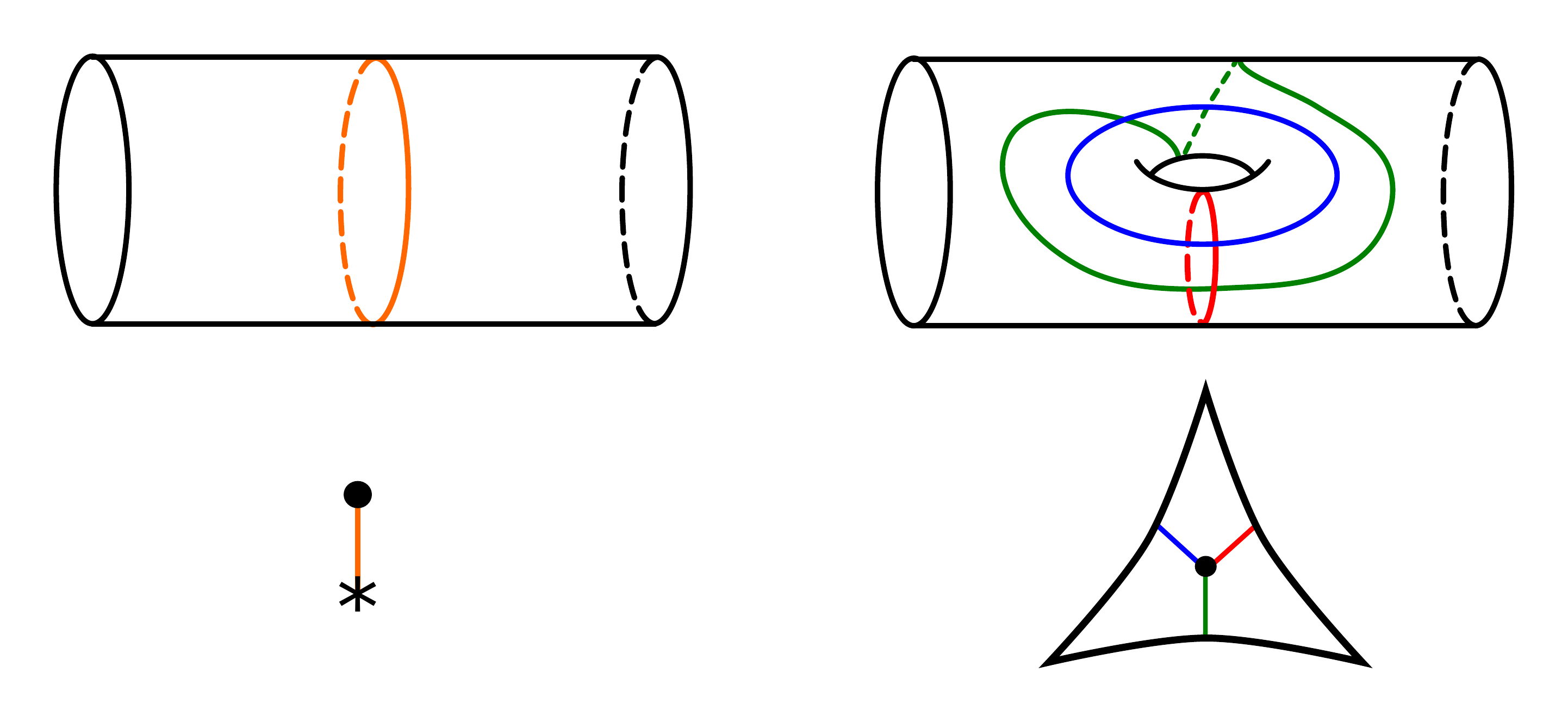}
    \caption{The resulting critical images and vanishing cycles when performing an unsink move.}
    \label{fig:wrinkle}
\end{figure}

Given a basepoint $b$, a critical point $c$, and two directed paths $p$ and $p'$ between $b$ and $c$, a homotopy from $p$ to $p'$ may cross over some cusps or Lefschetz singularities. Given some data about the Lefschetz singularities and the cusp singularities, one can deduce the relations between $c_p$ and $c_{p'}$. These changes in the vanishing cycles were explored in depth in \cite{Asa21} where this change is called a $\Delta_1$ move and \cite{Hay20} where this move was first analyzed.

\begin{lemma}(Lemma 3.2 of \cite{Hay20})
\label{lem:dragOverCusp}
Let $f: M^4 \to D^2$ be a wrinkled fibration, $b \subset D^2$ be a regular value, and $c$ be a critical value. Let $p$ and $p'$ be two paths between $b$ and $c$ related by a homotopy over a cusp as in Figure \ref{fig:referencePathsAroundCusp}. Let $c_1$, $c_2$, $c_3$, and $c_4$ by the vanishing cycles on $\Sigma_g$ of the folds in \ref{fig:referencePathsAroundCusp}. Then $c_4$ on $\Sigma_g$ is isotopic to $\tau_{\tau_{c_1}(c_2)}(c_3).$

\end{lemma}

\begin{figure}
    \centering
    \includegraphics[scale=.4]{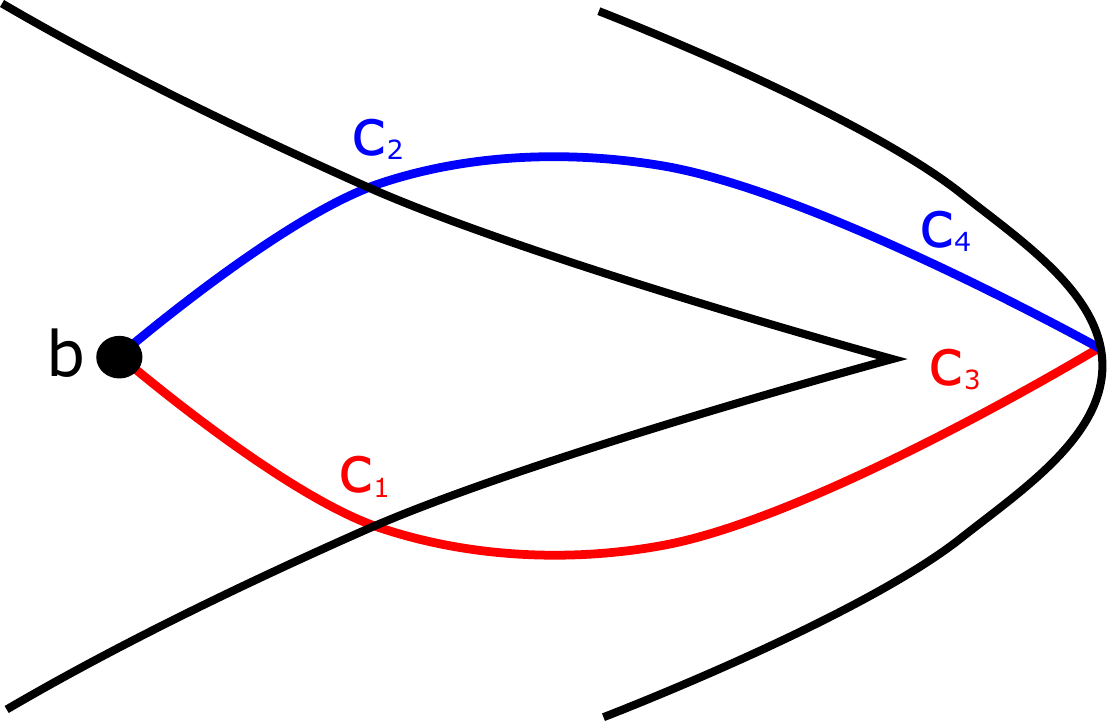}
    \caption{When passing a reference path over a cusp, the vanishing cycle $c_4$ on $\Sigma_g$ is isotopic to $\tau_{\tau_{c_1}(c_2)}(c_4).$}
    \label{fig:referencePathsAroundCusp}
\end{figure}

We note that this lemma is obtained by applying an unsink to the cusp to obtain a Lefschetz singularity with vanishing cycle $\tau_{c_1}(c_2)$ and using the well known monodromy of this Lefschetz singularity. 

\begin{figure}
    \centering
    \includegraphics[scale=.4]{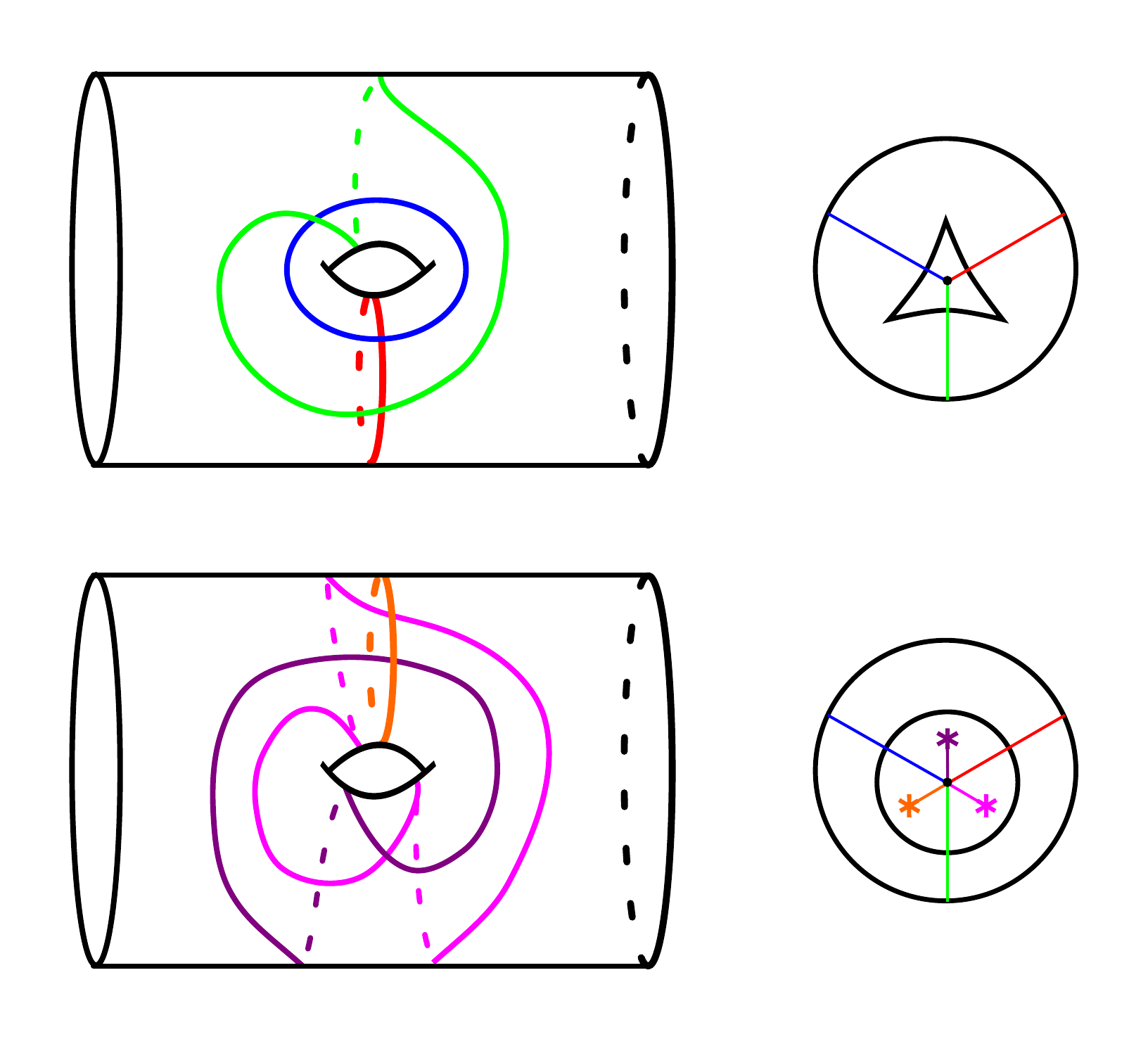}
    \caption{Top: The vanishing cycles of a tricuspidal curve resulting from wrinkling a Lefschetz singularity. Bottom: The vanishing cycles of the Lefschetz singularities obtained after unsinking the cusps.}
    \label{fig:twistCurves}
\end{figure}

\subsection{UPW move on diagrams}

In \cite{IN20} we define a composition of the unsink, push, and wrinkle moves which we call a UPW move, pictured in Figure \ref{fig:UPWWithPaths}. In this paper, we will additionally track reference paths in order to obtain a new multisection diagram. This will allow us to extract information about how the 3-dimensional handlebodies change under this operation in order to give corresponding moves in the dual handlebody complex. 

\begin{definition}
Let $F: X^4 \to D^2$ be a radially monotonic Morse 2-function, and let $c \subset D^2$ be a cusp value. We say that $F'$ is obtained by a \textbf{UPW move} on $c$ from $F$ if $F'$ is obtained from $F$ by unsinking the cusp $c$, dragging the resulting Lefschetz singularity towards the center of the disk until it has passed all other folds, and wrinkling the Lefschetz singularity. If $c$ lied in the sector $X_i$ of the corresponding multisection, then the cusps are distributed to the sectors $X_{i+1}$, $X_{i+2}$, and $X_{i+3}$.
\end{definition}

We next seek to describe the genus $g+1$ multisection diagram obtained after performing a UPW move on a genus $g$ multisection. This move begins by locating a cusp in the Morse 2-function, which corresponds to two curves $c_i$ and $c_{i+1}$ which intersect once on the corresponding multisection diagram. These two curves, together with the Lefschetz singularity obtained by unsinking the cusp lie in the punctured torus neighbourhood of $c_i \cup c_{i+1}$ as shown on the right of Figure \ref{fig:unsink}. We will only modify the diagram in this neighbourhood, as described in the following proposition.

\begin{proposition}
\label{prop:CurveEffect}
Let $D$ be a genus $g$ multisection diagram of a 4-manifold $X$ comprised of the ordered cut systems $(C_1, C_2, ..., C_n)$ on $\Sigma_g$ with the cut system $C_k$ being made up of the curves ${c_k^1, ... , c_k^g}$. Let $c_{i} \in C_{i}$ and $c_{i+1} \in C_{i+1}$ be two curves which intersect geometrically once. Then we may obtain a new multisection diagram $(D_1', D_2', ..., D_n')$ of $X$ by replacing the annular neighbourhood of $t_{c_i}(c_{i+1})$ in $D$ as in Figure \ref{fig:wrinklingAnnulus} and adding one new curve to each cut system as pictured. All curves which do not intersect $t_{c_i}(c_{i+1})$ are left unchanged.
\end{proposition}

\begin{proof}

Given a multisection diagram, we may construct an associated multisected Morse 2-function where the curves $c_i$ and $c_{i+1}$ are two curves coming from folds which meet in a cusp. Performing a UPW move on this cusp produces a new Morse 2-function with an inner fiber a surface of genus $g+1$. We first choose an identification of this surface, $\Sigma_{g+1}$, with our original multisection surface $\Sigma_g$. To do this, we take a path crossing through one of the indefinite folds which gives a diffeomorphism between the surface $\Sigma_{g+1}$ surgered along a curve $c$ and the surface $\Sigma_g$. This gives rise to reference paths as in the bottom left of Figure \ref{fig:UPWWithPaths}. We may choose this curve so that it is precisely the red curve in Figure \ref{fig:wrinkle} as viewed on $\Sigma_{g+1}$. In beginning to construct our new cut systems, we set ${d_l^{g+1}}  = c$ for $l \not \in \{i+2, i+3 \}$. We set ${d_{i+2}^{g+1}}$ to be the blue curve in Figure \ref{fig:wrinkle} and ${d_{i+3}^{g+1}}$ to be the green curve in Figure \ref{fig:wrinkle}.

Following these reference paths from $\Sigma_{g+1}$, all of the curves on $\Sigma_g$ coming from the indefinite fold singularities are determined on $\Sigma_{g+1}$ up to handle slides over $c$. Our chosen identification of $\Sigma_g$ as $\Sigma_{g+1}$ surgered along $c$ also requires the images of the curves in $\Sigma_{g+1}$ to be disjoint from $c$. Then we may choose the initial image of the curves under these reference paths as in Figure \ref{fig:initialAnnulus} where all of the curves move across the top of the new genus, except for the curves in $C_{i+3}$ which are made to avoid the new curve ${d_{i+3}^{g+1}}$ which is to be added to $C_{i+3}$.

In order to get the radial reference paths for a multisected Morse 2-function, must modify the reference paths for the handlebodies $C_{i+2}$ and $C_{i+3}$ by dragging them over cusps as in the transition between the bottom two configurations of Figure \ref{fig:UPWWithPaths}. Lemma \ref{lem:dragOverCusp} allows us to track how the curves change under this procedure. When dragging the curves of $C_{i+2}$ and $C_{i+3}$ over the first cusp we must twist these curves over the purple curve in Figure \ref{fig:twistCurves}. Since the curves of $C_{i+2}$ are disjoint from this curve, they are left unaffected by this process. On the other hand, the portions of the curves corresponding to the cut system for $C_{i+3}$ change from the green arcs in Figure \ref{fig:initialAnnulus} to the green arcs in Figure \ref{fig:wrinklingAnnulus}. We must then further twist the curves of $C_{i+3}$ over the orange curve in Figure \ref{fig:twistCurves}, however the curves are disjoint, so that we have already completed the process for obtaining the cut system $D$. 

\end{proof}

\begin{figure}
    \centering
    \includegraphics[scale =.35]{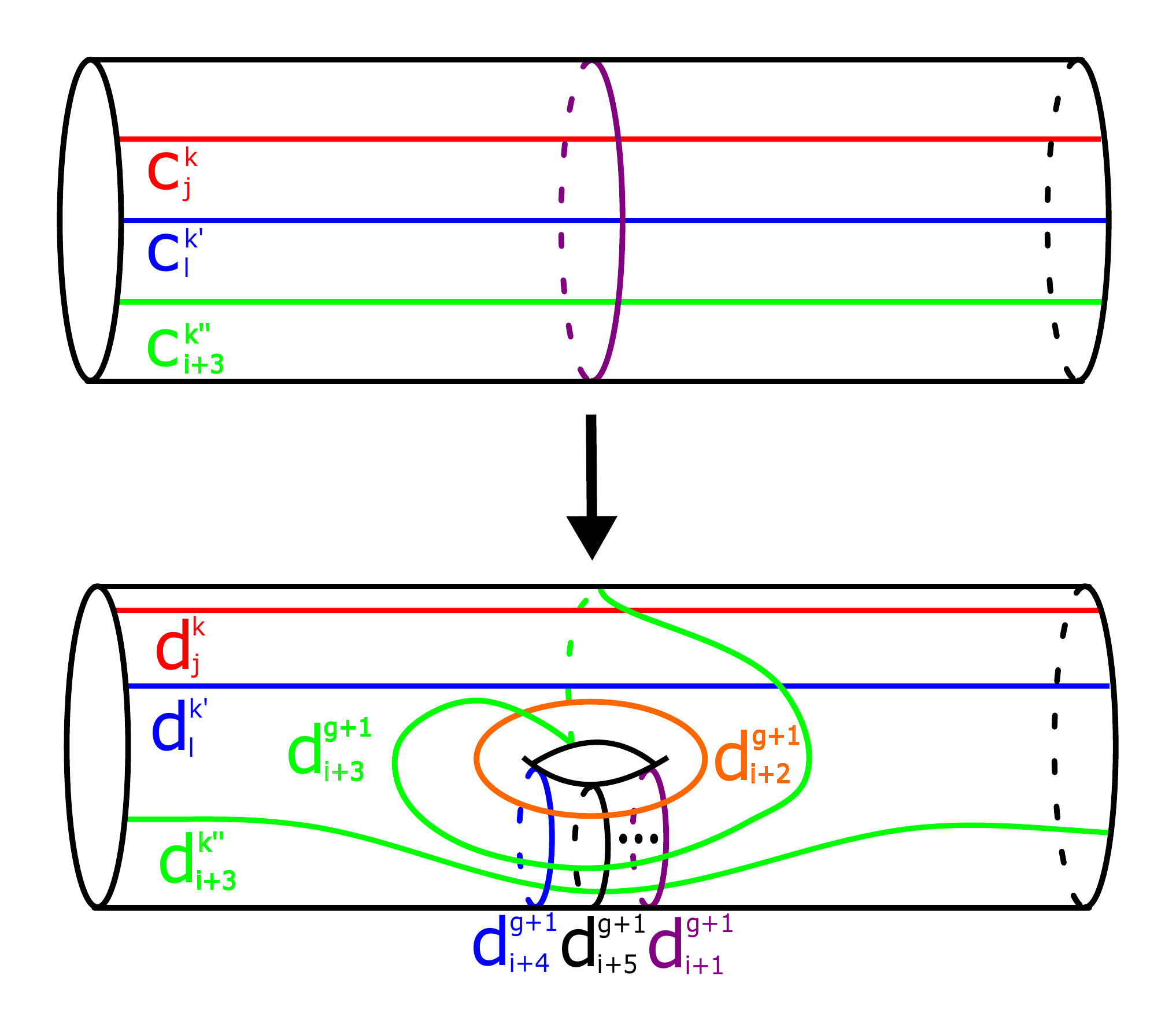}
    \caption{The curves on a multisection diagram after wrinkling a Lefschetz singularity whose vanishing cycle is the purple curve above and redrawing reference paths. All curves in the cut system $C_k$ for $k\neq i+3$ are sent ``over" the new genus, whereas any curves in the cut system $C_{i+3}$ are sent ``below" the new genus.}
    \label{fig:wrinklingAnnulus}
\end{figure}

\begin{figure}
    \centering
    \includegraphics[scale=.3]{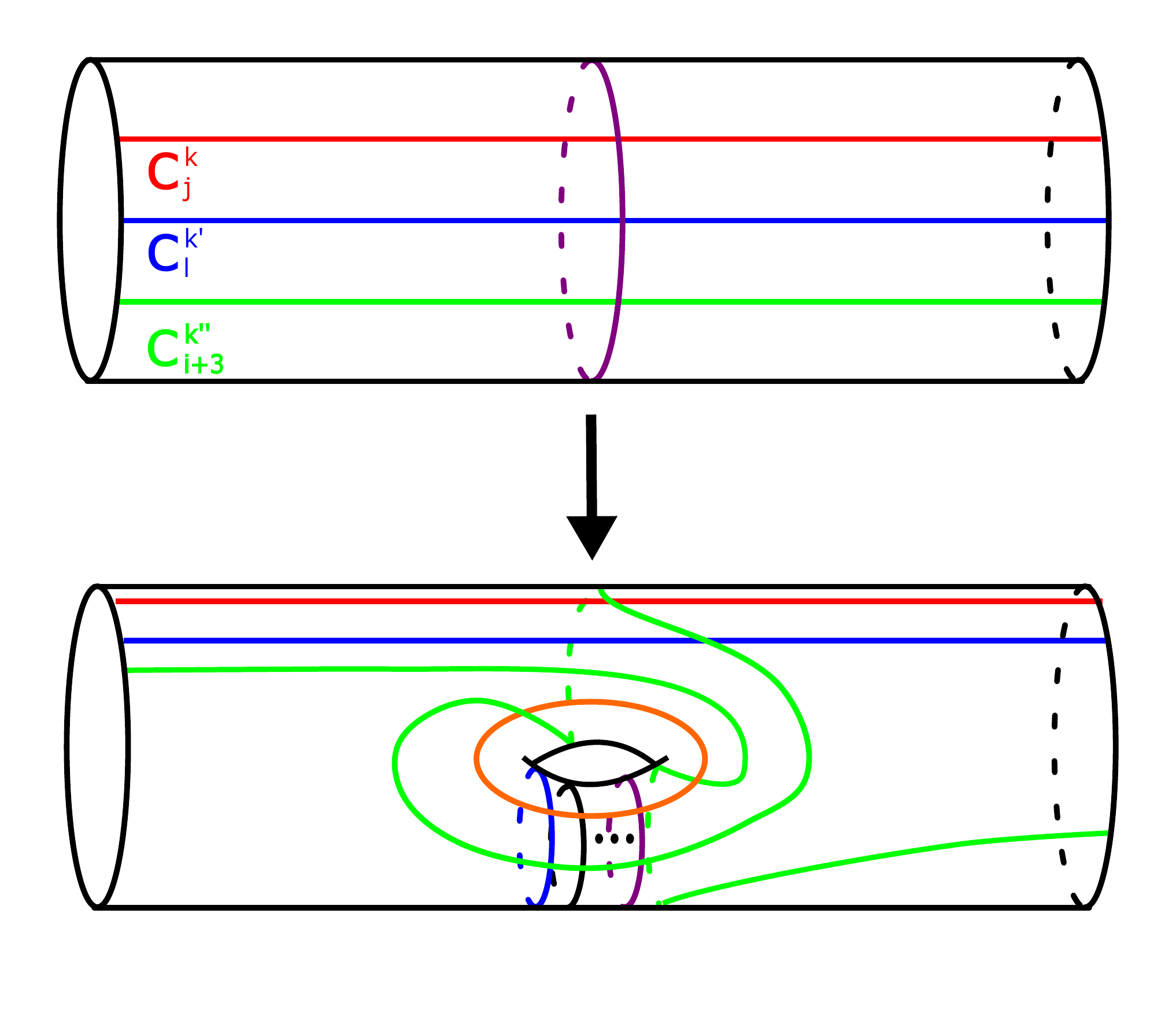}
    \caption{The initial configuration of curves in the proof of Proposition \ref{prop:CurveEffect}. These curves will not in general give a multisection diagram as the reference paths need to be corrected. }
    \label{fig:initialAnnulus}
\end{figure}

\begin{figure}
    \centering
    \includegraphics[scale = .35]{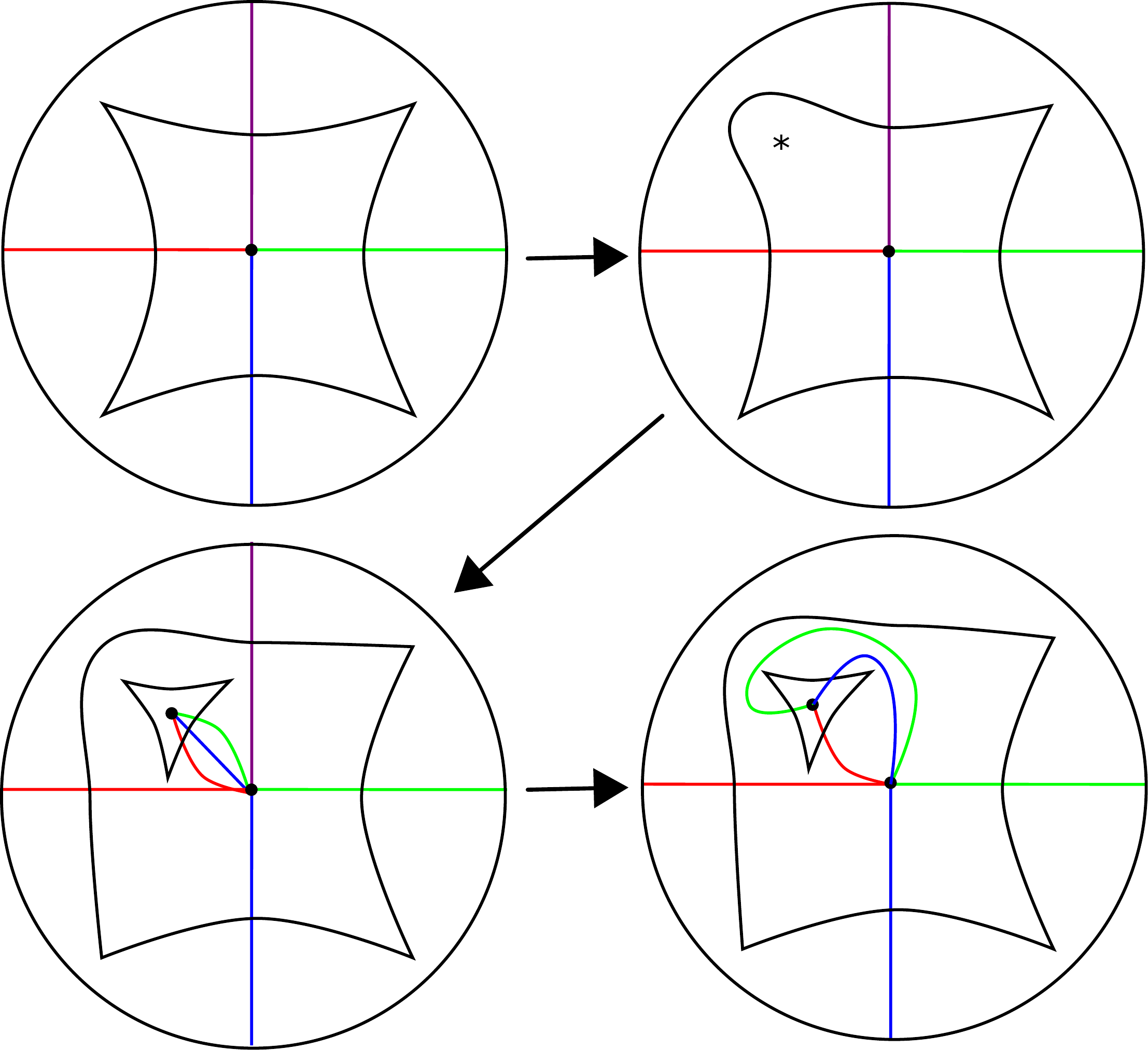}
    \caption{A UPW move on a 4-section along with steps for redrawing reference paths for the resulting trisection.}
    \label{fig:UPWWithPaths}
\end{figure}

\section{Stabilization and second bijections}
\label{sec:Stablization}

In this section, we seek to prove a bijection between the set of smooth, orientable, compact 4-manifolds up to diffeomorphism and the sets in Theorem \ref{thm:firstBijections}, modulo some additional relations including a UPW move. Our strategy is to first use a UPW move to turn any multisection into a trisection and then to show that a UPW move on inner tricusps realizes the trisection stabilization move of \cite{GK}. The first part was shown in Proposition 8.4 of \cite{IN20}. In particular, we show the following.

\begin{proposition} (\cite{IN20} Proposition 8.4)
\label{prop:decreaseSectors}
Let $n>3$, and suppose a 4-manifold $X$ admits an $n$-section of genus $g$ with sectors $X_1,\dots,X_n$. Then, $X$ admits an $(n-1)$-section of genus $(2g-k_1)$, where $X_1\cong \natural^{k_1}S^1\times B_3$ obtained by a UPW move.
\end{proposition}

Repeatedly applying this operation to a multisection reduces the number of sectors until we are left with only three. An example of this process is shown in Figure \ref{fig:S2xS2QuadtoTri} where a 4-section of $S^2 \times S^2$ is made into the unique genus 2 trisection of $S^2 \times S^2.$ Figure \ref{fig:S2xS2slidesToStandard} shows a sequence of transformations turning the resulting trisection diagram into the familiar trisection diagram for $S^2 \times S^2$. 

\begin{figure}
    \centering
    \includegraphics[scale=.4]{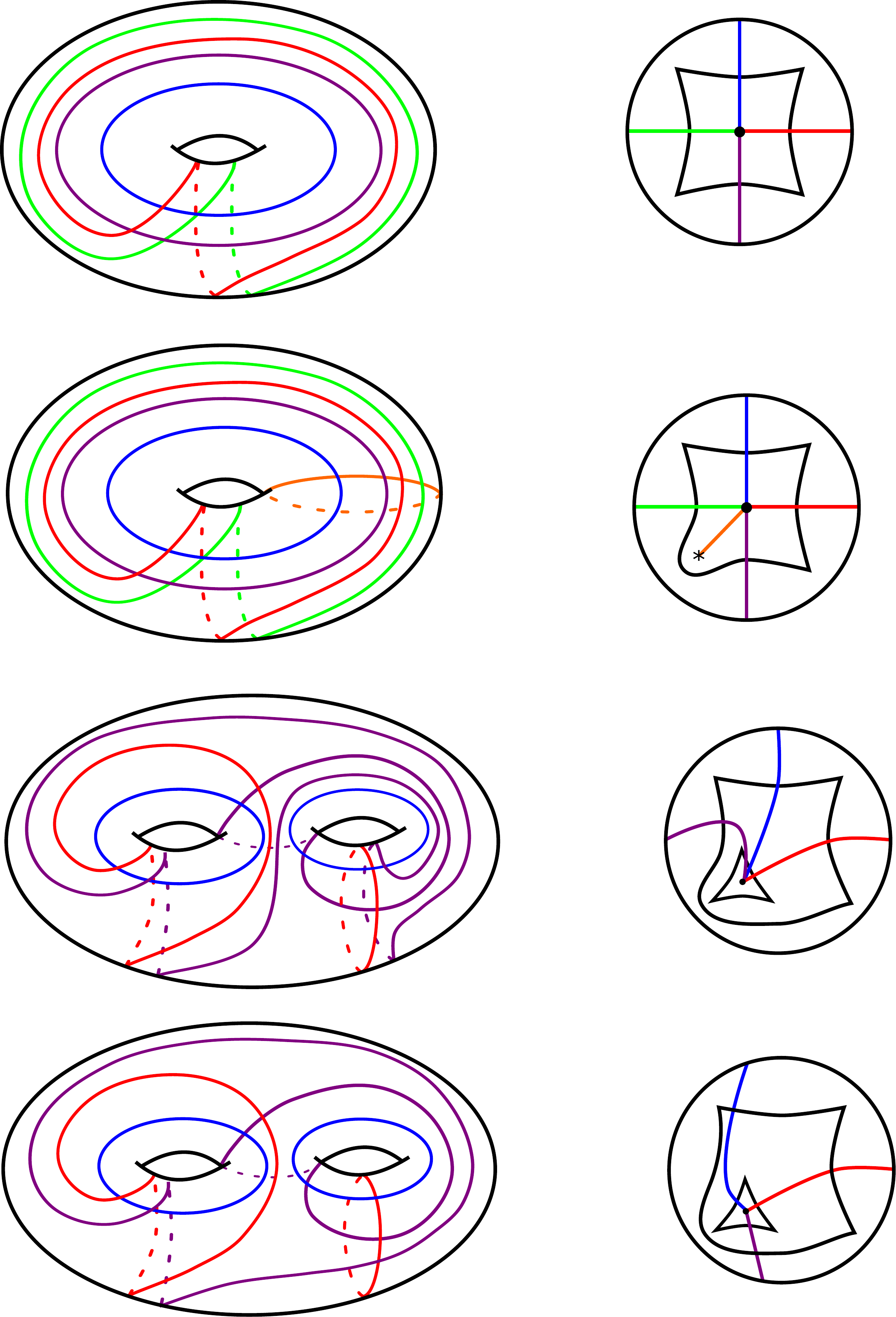}
    \caption{Performing a UPW move on a 4-section of $S^2 \times S^2$. Critical images of the corresponding wrinkled fibrations are pictured to the right. The final picture is a trisection diagram for of $S^2 \times S^2$ which is transformed into the usual picture in Figure \ref{fig:S2xS2slidesToStandard}}.
    \label{fig:S2xS2QuadtoTri}
\end{figure}

\begin{figure}
    \centering
    \includegraphics[scale=.3]{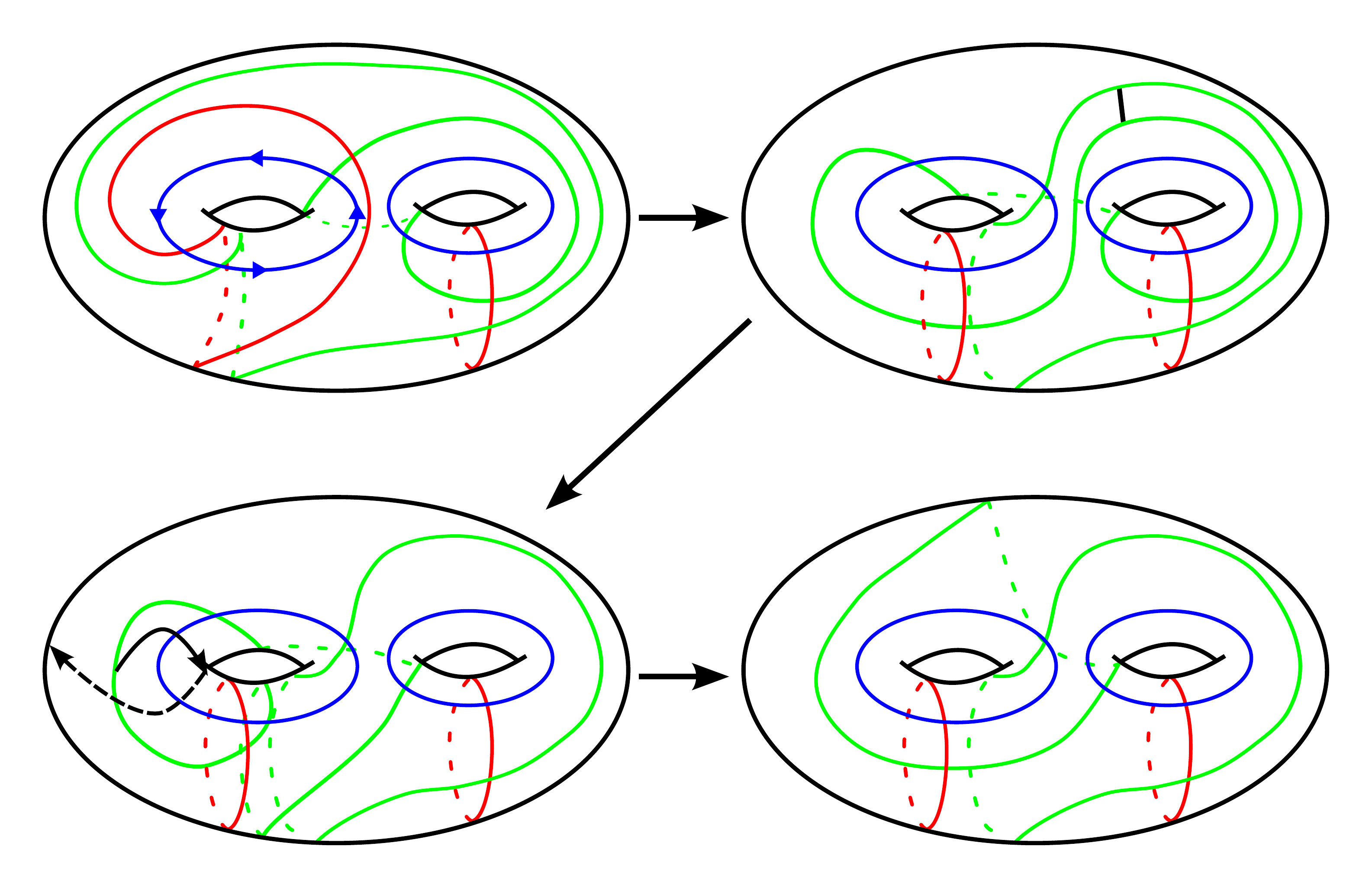}
    \caption{A sequence of operations transforming the trisection diagram of $S^2 \times S^2$ in Figure \ref{fig:S2xS2QuadtoTri} to the usual picture. The purple curves in the previous figure are colored green here to agree with trisection conventions. We first twist about the blue curve in the first picture. A handle slide of the green curves along the band pictured produces the diagram in the bottom left. The resulting green curve can be isotoped into the picture in the bottom right.}
    \label{fig:S2xS2slidesToStandard}
\end{figure}

In \cite{MSZ} the authors provide a generalization of the trisection stabilization operation called an unbalanced stabilization. This operation breaks the trisection stabilization operation into three sub-operations called $1-$ $2-$ or $3-$ stabilization. These operations are characterized by finding two parallel curves in different handlebodies which intersect a curve in the third handlebody exactly once. Here, we show that performing a UPW move on the cusp in the $i^{th}$ sector of an inner tricusp produces an $i-$stabilization. Performing this three times gives us the desired trisection stabilization operation.

\begin{proposition}
\label{prop:UPWisStab}
Let $F:X^4 \to D^2$ be a trisected Morse 2-function whose innermost fold circle is embedded and contains exactly three cusps. Let $M(F) = X_1 \cup X_2 \cup X_3$ be the trisection corresponding to $F$. Then performing a UPW move on the cusp lying in $X_i$ corresponds to an $i-$stabilization.
\end{proposition}

\begin{proof}
The three vanishing cycles on the inner fold each intersect pairwise once and fill up a thrice punctured torus. The curves separate the surface into three pieces each containing a single boundary component. There are two cases for the configuration of these three vanishing cycles up to orientation preserving homeomorphism rel boundary which are shown in Figure 6 of \cite{Hay20}. These are characterized by whether the curves can be oriented to pairwise intersect positively once. We will treat the case in which these curves intersect positively once, with the other case being identical after composing with an orientation reversing homeomorphism.

In this case, the curves corresponding to the vanishing cycles are shown on the left of Figure \ref{fig:UPWonInnerTricusp}. Performing a UPW move between the red and the blue curve yields the curves on the right of Figure \ref{fig:UPWonInnerTricusp}. Here, the curves are drawn in slightly different shades only for visual clarity. 

In Figure \ref{fig:sildesToStabilization} we perform a sequence of slides and isotopies on the resulting curves. At the end of this sequence we find a parallel red and blue curve which intersects the lighter green curve in the top right exactly once indicating that this is a stabilized trisection. In addition all of the other curves are disjoint from this configuration. Removing a regular neighbourhood of the red and green curves leaves a solid torus with a new boundary component which may be filled in with a disk to produce a trisection of the destabilized manifold.

In Figure \ref{fig:remainingCurves} we see the remaining curves of this trisection. The orientations on the left surface show that the curves may be oriented to pairwise intersect positively once. The curves separate the region into the three shaded regions shown on the right of Figure \ref{fig:remainingCurves}, each containing one boundary component. This shows that the configuration is homeomorphic to the configuration we started with, so that this process did indeed give an unbalanced stabilization of our original trisection
\end{proof}

\begin{figure}
    \centering
    \includegraphics[scale=.5]{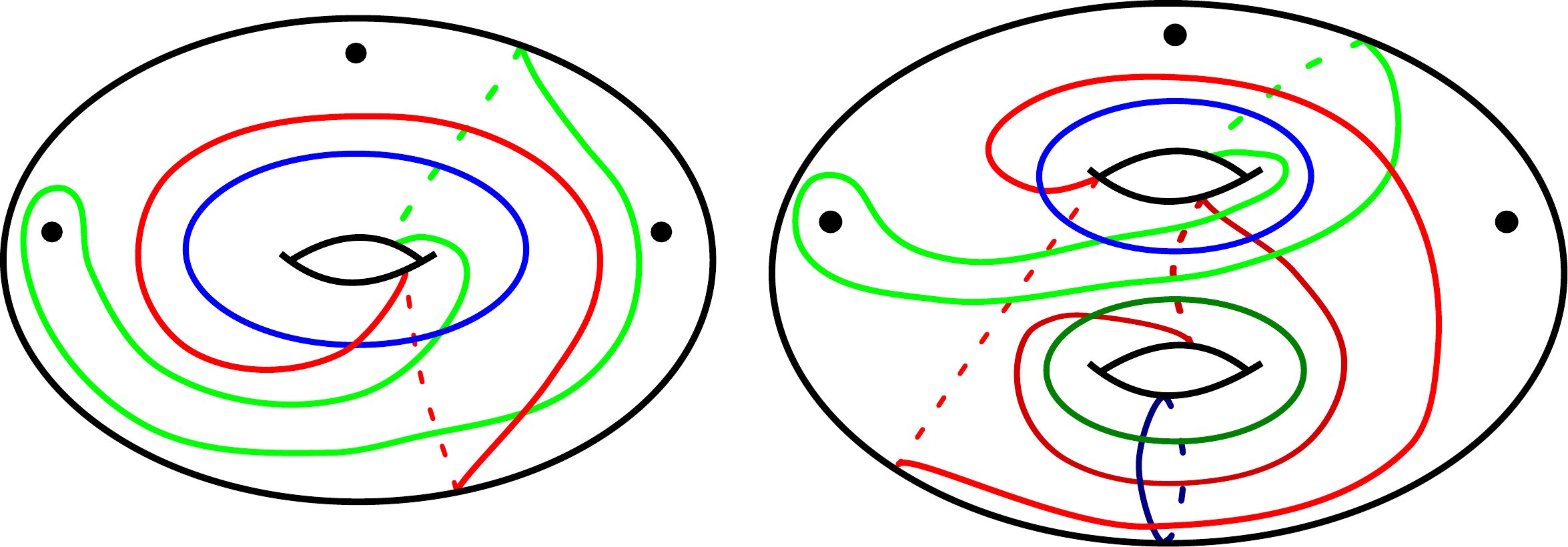}
    \caption{Left: The three vanishing cycles on an inner tricusp. Right: The diagram obtained by performing a UPW move between the red and the blue curves. The three new curves are colored a slightly darker shade of the color of the handlebody they belong to for visual clarity.}
    \label{fig:UPWonInnerTricusp}
\end{figure}

\begin{figure}
    \centering
    \includegraphics[scale=.43]{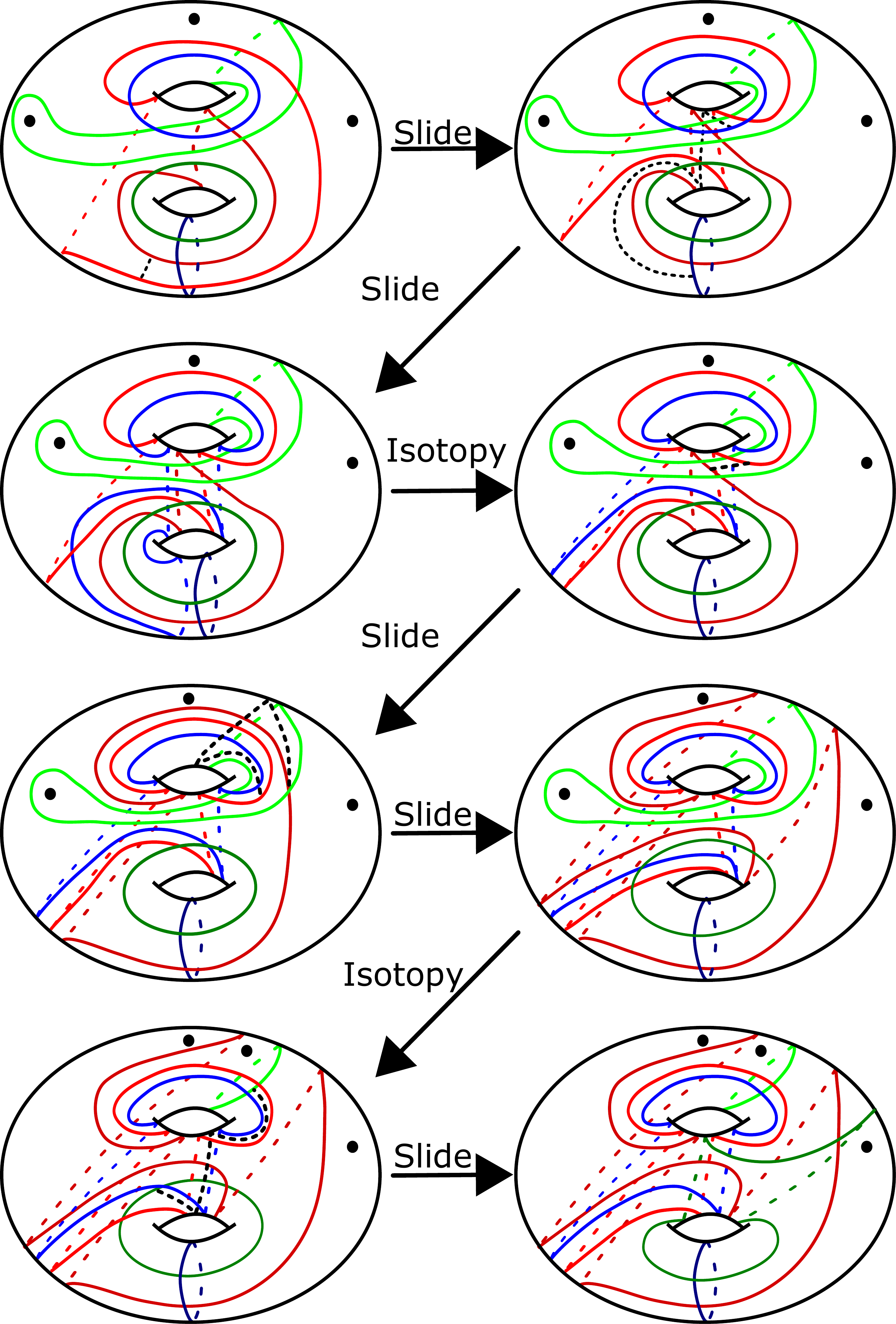}
    \caption{A sequence of slides to show that the multisection in Figure \ref{fig:UPWonInnerTricusp} is stabilized. Handle slides are done along the dotted black arcs.}
    \label{fig:sildesToStabilization}
\end{figure}

\begin{figure}
    \centering
    \includegraphics[scale=.43]{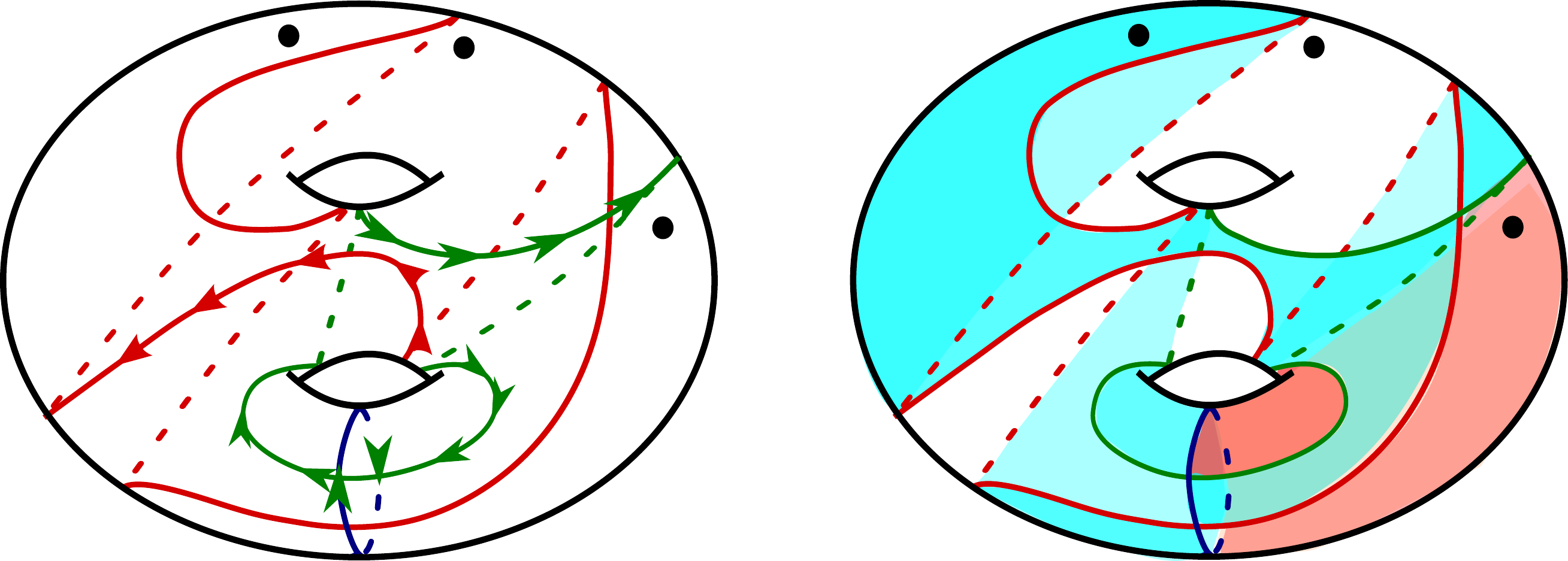}
    \caption{The three curves remaining after destabilizing. With the orientations on the left, each pair of curves intersect once positively. The curves separate the surface into three components, each of which contains one boundary component. These two properties characterize the curves up to orientation preserving homeomorphism fixing the marked points.}
    \label{fig:remainingCurves}
\end{figure}

We are now ready to prove our stable equivalence theorem for multisections.

\begin{theorem}

Any two multisections of a given 4-manifold, $X^4$, become diffeomorphic after a sequence of UPW moves and contractions. Any two multisection diagrams of $X^4$ are related by handle slides of cut systems, isotopies of cut systems, UPW moves, contractions and an overall diffeomorphisms of the multisection surfaces.
\end{theorem}

\begin{proof}
Let $T$ and $T'$ be two multisections of $X^4.$ By Proposition \ref{prop:decreaseSectors} we can take sequences of UPW moves and contractions to reduce $T$ and $T'$ to trisections. Furthermore, by performing an additional UPW move we may assume that these trisections have Morse 2-functions which have an embedded inner fold with three cusps. By Proposition \ref{prop:UPWisStab} we can perform UPW moves to realize unbalanced stabilizations. This move also leaves another embedded inner fold with three cusps, so that we may repeat this operation in each sector to produce a balanced stabilization. By \cite{GK}, any two trisections are related by a sequence of balanced stabilizations. As we have realized this operation by UPW moves, we have proven the first statement of the theorem. After the first statement is known, the second statement follows exactly as in the case of trisections \cite{GK}.
\end{proof}

In order to provide a correspondence between (not necessarily thin) multisections and the previously defined objects we require one additional relation. We first define this relation on loops in the cut complex. Recall that a type 1 move in $C(\Sigma)$ is supported in the punctured torus filled by the two curves intersecting once.

\begin{definition}
Let $C$ be a fixed cut system on $\Sigma_g$. Two type 1 moves in $C(\Sigma)$ originating from $C$ are said to be \textbf{commuting type-1 moves} if the moves are supported in disjoint subsurfaces.
\end{definition}

If two type 1 moves represented by edges $e_1:C_1 \to C_2$ and $e_1': C_1 \to C_2'$ commute, then the resulting cut systems $C_2$ and $C_2'$ are both connected to a third cut system $C_3$ obtained by performing both curve replacements. We call these edges $e_2$ and $e_2'$ respectively, and this configuration is shown in Figure \ref{fig:commutingType1}. 

We say loops $L$ and $L'$ in $C(\Sigma)$ are related by a \textbf{commutation move} if $L'$ is obtain from $L$ by replacing edges of the form $e_1$ and $e_2$ by edges $e_1'$ and $e_2'$. This corresponds to replacing a portion of the loop which, up to homeomorphism, looks like the top two edges of Figure \ref{fig:commutingType1} with the bottom two edges in Figure \ref{fig:commutingType1}. We also say that two loops $L$ and $L'$ in $D(\Sigma)$ are related by a commutation move if they are the images of loops in $C(\Sigma)$ which are related by a commutation move.

\begin{figure}
    \centering
    \includegraphics[scale=.2]{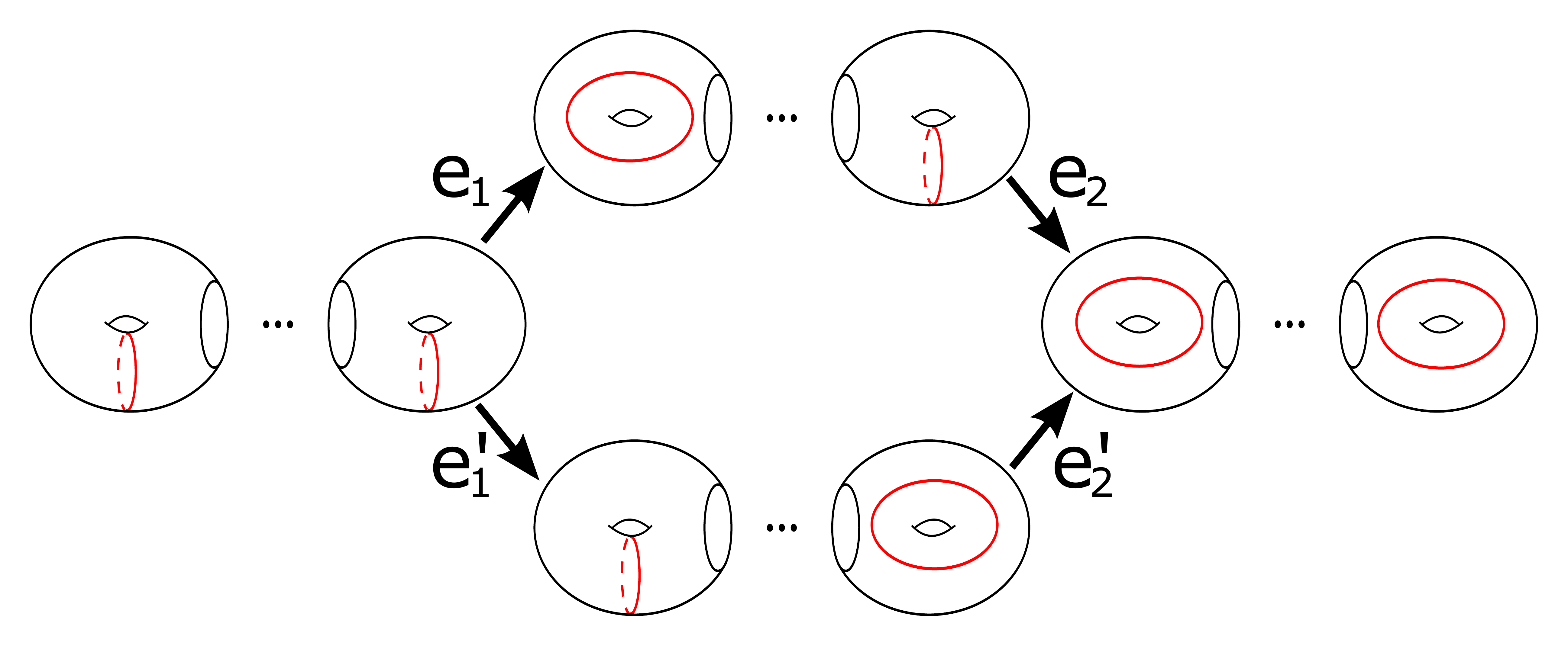}
    \caption{A commutation move is given by replacing the top two edges in a path with the bottom two edges.}
    \label{fig:commutingType1}
\end{figure}

This final relation allows us to take a multisection and obtain a loop in the dual handlebody complex modulo commutation relations.

\begin{lemma}
There is a bijection between genus $g$ multisections of 4-manifolds, and loops in $D(\Sigma_g)$ modulo commutation moves.
\end{lemma} 

\begin{proof}
Recall that any multisection can be made thin after a sequence of expansions. If a multisection is not thin, then it has some sector $X_i$ with $X_i \cong \sharp^k S^1 \times B^3$ with $k < g-1$. The Heegaard splitting on $\partial X_i$ given by $H_{i-1,i} \cup H_{i,i+1}$ is a Heegaard splitting of $\#^{k}S^1 \times S^2$. By Waldhausen's theorem \cite{FW}, there is a Heegaard diagram for this manifold consisting of $k$ parallel curves and $g-k$ pairs of curves which intersect once. Different choices of thinning the multisection correspond, in a multisection diagram, to choices of orders in which to flip these curves which intersect once. The resulting thin multisections depend on these choices but when passing to a loop in the dual handlebody complex, these choices correspond to commutation moves. The result then follows from Proposition \ref{prop:loopsAndThinMultisections}.
\end{proof}

By passing through the bijections in Theorem \ref{thm:firstBijections}, we also obtain a natural commutation relation on generic paths of Morse functions up to equivalence. Here commuting paths correspond to moving two critical value swaps in two disjoint copies of $S_{1,2}$ past each other in time. We next define the sets involved in our main theorem. Note that A UPW move takes an object associated to a genus $g$ surface and assigns to it an object on a genus $g+1$ surface.

\begin{definition}
We denote by $\mathbf{\tilde{L}(D)}$ the set of loops in all dual handlebody complexes modulo UPW moves and commutation moves. We denote by $\mathbf{\tilde{L}(\mathfrak{M})}$ the set $\mathfrak{M_h}$ modulo UPW moves and commutation moves. We denote by $\mathbf{\tilde{T}}$ the set of thin multisections of smooth 4-manifolds modulo expansion, contraction, and UPW moves. We denote by $\mathfrak{4man}$ the set of smooth, orientable, closed 4-manifolds up to diffeomorphisms.
\end{definition}

The following theorem summarizes the results of this paper.

\begin{theorem}
\label{thm:secondBijections}
There are bijections $\tilde{L}(D) \leftrightarrow \tilde{L}(\mathfrak{M})  \leftrightarrow \tilde{T}(\Sigma_g) \leftrightarrow \mathfrak{4man}$.
\end{theorem}

\bibliography{thebib.bib}
\bibliographystyle{plain}

\end{document}